\documentclass[a4paper]{amsart}

\usepackage{amsmath,amsthm,amssymb,enumerate}
\usepackage{epsfig}
\usepackage{amssymb}
\usepackage{amsmath}
\usepackage{amssymb}
\usepackage{amsmath,amsthm}
\usepackage[latin1]{inputenc}
\usepackage[T1]{fontenc}
\usepackage{path}
\usepackage{ae,aecompl}
\usepackage{amsfonts}
\usepackage{amsxtra}
\usepackage{euscript,mathrsfs}
\usepackage{color}
\usepackage[left=3.3cm,right=3.3cm,top=4cm,bottom=3cm]{geometry}
\usepackage[citecolor=blue,colorlinks=true]{hyperref}
\allowdisplaybreaks

\usepackage{enumitem}
\setenumerate{label={\rm (\alph{*})}}

\usepackage{amsfonts}
\usepackage{amsxtra}

\numberwithin{equation}{section}

\newcommand{\bfu}{\vc{u}}

\newcommand{\bfq}{\vc{q}}

\newcommand{\bfvarphi}{\boldsymbol{\varphi}}

\newcommand{\pas}{\prst\mbox{-a.s.}}

\newcommand{\D}{{\rm d}}

\newcommand{\intQ}[1]{\int_{Q} #1 \ \dx}
\newcommand{\Q}{\mathbb{T}^3}

\newcommand{\vr}{\varrho}

\newcommand{\vu}{\vc{u}}

\newcommand{\tmop}[1]{{\mathrm #1}}
\newcommand{\tmmathbf}[1]{{\bf #1}}
\newcommand{\vc}[1]{{\bf #1}}

\newcommand{\expe}[1]{ \mathbb{E} \left[ #1 \right] }

\newcommand{\Div}{\divergence}
\newcommand{\ep}{\varepsilon}
\newcommand{\R}{\mathbb R}
\newcommand{\N}{\mathbb N}

\newcommand{\E}{\mathbb E}

\newcommand{\dH}{\,\mathrm{d}\mathcal H^2}
\newcommand{\dd}{\mathrm{d}}
\newcommand{\dx}{\,\mathrm{d}x}
\newcommand{\dt}{\,\mathrm{d}t}
\newcommand{\dxt}{\,\mathrm{d}x\,\mathrm{d}t}

\newcommand{\dif}{\mathrm{d}}

\newcommand{\mf}{\mathfrak{F}}

\newcommand{\prst}{\mathbb{P}}

\newcommand{\mn}{\mathbb{N}}

\newcommand{\mt}{Q}

\DeclareMathOperator{\diver}{div}
\DeclareMathOperator{\divergence}{div}

\newcommand{\bFormula}[1]{
\begin{equation} \label{#1}}
\newcommand{\eF}{\end{equation}}

\newcommand{\Ov}[1]{\overline{#1}}

\newcommand{\DC}{C^\infty_c}

\newcommand{\aleq}{\stackrel{<}{\sim}}
\newcommand{\ageq}{\stackrel{>}{\sim}}

\newcommand{\vre}{\vr_\ep}

\newcommand{\vue}{\vu_\ep}

\newcommand{\tvr}{\tilde \vr}

\newcommand{\cal}{\mathcal}

\newcommand{\vt}{\vartheta}
\newcommand{\vte}{\vt_\ep}
\newcommand{\bfU}{\tilde \vu}

\newcommand{\bfS}{\mathbb S}
\newcommand{\Grad}{\nabla}

\newcommand{\intO}[1]{\int_{\mt} #1 \ \dx}

\newcommand{\intTor}[1]{\int_{Q} #1 \ \dx}

\newtheorem{Theorem}{Theorem}[section]

\newtheorem{Proposition}{Proposition}[section]

\newtheorem{Lemma}{Lemma}[section]
\newtheorem{corollary}{Corollary}[section]

\newtheorem{definition}{Definition}[section]
\newtheorem{Definition}{Definition}[section]

\allowdisplaybreaks

\begin{document}


\title{Stationary solutions in thermodynamics of stochastically forced fluids}

\author{Dominic Breit}
\address[D. Breit]{
Department of Mathematics, Heriot-Watt University, Riccarton Edinburgh EH14 4AS, UK}
\email{d.breit@hw.ac.uk}

\author{Eduard Feireisl}
\address[E. Feireisl]{Institute of Mathematics AS CR, \v{Z}itn\'a 25, 115 67 Praha 1, Czech Republic
\and Institute of Mathematics, TU Berlin, Strasse des 17.Juni, Berlin, Germany }
\email{feireisl@math.cas.cz}
\thanks{The research of E.F. leading to these results has received funding from
the Czech Sciences Foundation (GA\v CR), Grant Agreement
21--02411S. The Institute of Mathematics of the Academy of Sciences of
the Czech Republic is supported by RVO:67985840. The stay of E.F. at TU Berlin is supported by Einstein Foundation, Berlin.}

\author{Martina Hofmanov\'a}
\address[M. Hofmanov\'a]{Fakult\"at f\"ur Mathematik, Universit\"at Bielefeld, D-33501 Bielefeld, Germany}
\email{hofmanova@math.uni-bielefeld.de}
\thanks{M.H. gratefully acknowledges the financial support by the German Science Foundation DFG via the Collaborative Research Center SFB1283.}

\begin{abstract}
We study the full Navier--Stokes--Fourier system governing the motion of a general viscous, heat-conducting, and compressible fluid subject to
stochastic perturbation. The system is supplemented
with non-homogeneous Neumann boundary conditions
for the temperature and hence energetically open. We show that, in contrast with the energetically closed 
system, there exists a stationary solution.
Our approach is based on new global-in-time estimates which rely on the non-homogeneous boundary conditions combined with estimates for the pressure.

\end{abstract}

\subjclass[2010]{60H15, 35R60, 76N10,  35Q35}
\keywords{Compressible fluids, stochastic Navier--Stokes--Fourier system, stationary solution}

\date{\today}

\maketitle


\section{Introduction}
It is a common believe that the behaviour of turbulent fluid flows can be fully characterized
by a \emph{steady state} of the system (driven by a suitable stochastic forcing to substitute for possible perturbations due to changes in the boundary data), which is approached asymptotically for large times. 
Mathematically speaking this gives rise to an \emph{invariant measure} of the underlying system.
This is well-understood for the 2D incompressible stochastic Navier--Stokes equations, cf. \cite{francoa,FlMa,Ku,HM}, where uniqueness is well-known.

If uniqueness is not at hand, even the definition of an invariant measure becomes ambiguous, and one rather studies \emph{stationary solutions} of the dynamics: solutions with a probability law which does not change in time. This law serves as a substitute for an invariant measure.
The existence of stationary solutions to the 3D incompressible stochastic Navier--Stokes equations
is a nowadays classical result from \cite{franco}. More recently a counterpart for the
compressible stochastic Navier--Stokes equations has been established in \cite{BFHM}. It is interesting to note that in both cases stationarity provides a certain regularising effect on the solutions (see also \cite{FlaRom1} in connection with this).

One may think that adding further physical principles such as the possibility of \emph{heat transfer} completes the picture. The stochastic Navier--Stokes--Fourier equations haven been studied in \cite{BF} and the existence of weak martingale solutions has been shown. They describe the motion of a general viscous, heat-conducting, and compressible fluid subject to
stochastic perturbation based on the \emph{Second Law of Thermodynamics} via an entropy balance
as in \cite{F} (see also \cite{SmTr} for an alternative approach based on the internal energy balance due to \cite{fei3}).
Supplemented  with homogeneous Neumann boundary conditions for the temperature this is an \emph{energetically  closed} system.
The mechanical energy which is lost as dissipation is transfered into heat and, different to the incompressible or the isentropic Navier--Stokes equations, weak solutions are known to satisfy an energy equality.
The latter one shows that
the noise is constantly adding energy to the system such that it can never reach a steady state and, as shown in \cite[Section 7]{BF}, stationary solutions do not exist. Since this is physically not acceptable we are looking for a physical principle which can counteract the energy creation by the noise.\\

Different to \cite{BF} we consider in this paper an \emph{energetically open} version of the
stochastic Navier--Stokes--Fourier equations, where heat can drain through the boundary, see
\eqref{beq:1} below.
The time evolution of the fluid in the reference physical domain $Q\subset R^3$
is governed by the following set of equations:
\begin{subequations}\label{eq:1}
\begin{align}
\dd\varrho+\Div (\varrho\bfu)\dt&=0 ,\label{eq:12}\\
\dd(\varrho\bfu)+ \left[ \Div(\rho\bfu\otimes\bfu)  + \Grad p(\vr, \vt) \right] \dt &= \Div \bfS (\vt, \Grad \vu) \dt+\varrho{\mathbb{F}}(\varrho,\vt,\bfu)\,\dd W,\,\,\label{eq:11}\\
\dd(\varrho e(\vr, \vt))&+\big[\Div(\varrho e(\vr, \vt) \bfu)+\Div\bfq (\vt, \Grad \vt) \big]\dt\label{eq:13}\\&=\big[\bfS(\vt, \Grad \vu) :\nabla\bfu-p(\vr, \vt) \Div\bfu\big]\dt ,
 \nonumber
\end{align}
\end{subequations}
where $W$ is a cylindrical Wiener process and the diffusion coefficient $\mathbb{F}$ can be identified with a sequence $(\mathbf{F}_{k})_{k\geq 1}$ satisfying a suitable Hilbert-Schmidt assumption, see Section~\ref{sec:framework} for the precise definitions.
Here $\varrho$ denotes the density of the fluid, $\vt$ the absolute temperature and $\bfu$ the velocity field.
For the viscous stress tensor we suppose
Newton's rheological law
\begin{align}\label{eq:nr}
\bfS=\bfS(\vartheta,\nabla \bfu)=\mu(\vartheta)\Big(\nabla\bfu+\nabla\bfu^T-\frac{2}{3}\Div\bfu\,\mathbb I\Big)+\eta(\vartheta)\Div\bfu\,\mathbb I.
\end{align}

The internal energy (heat) flux is determined by Fourier's law
\begin{align}\label{eq:fl}
\bfq=\bfq(\vartheta,\nabla\vartheta)=-\kappa(\vartheta)\nabla\vartheta=-\nabla\mathcal K(\vartheta),\quad \mathcal K(\vartheta)=\int_0^\vartheta\kappa(z)\,\dd z.
\end{align}
The thermodynamic functions $p$ and $e$ are related to the (specific) entropy $s = s(\vr, \vt)$ through {Gibbs' equation}
\bFormula{m97} \vartheta D s (\varrho, \vartheta) = D e (\varrho,
\vartheta) + p (\varrho, \vartheta) D \Big( \frac{1}{\varrho} \Big)
\ \mbox{for all} \ \varrho, \vartheta > 0,
\eF
\\
where $D$ denotes the total derivative with respect to $(\vr,\vt)$.
We supplement \eqref{eq:1}--\eqref{m97} with the boundary conditions (see also \cite{FMNP})
\begin{equation} \label{beq:1}
\vu|_{\partial Q} = 0,\ \bfq  \cdot \vc{n}|_{\partial Q} = d(\vt)(\vt-\Theta_0),\ \mbox{and fix the total mass}\ \int_Q\vr\dx=M_0,
\end{equation}
where $\Theta_0\in L^1(\partial Q)$ is strictly positive, $M_0>0$ and we suppose that there are $\underline d,\overline d>0$ such that
\begin{align}\label{beq:2}
\underline d\vt\leq d(x,\vt)\leq \overline d \vt\quad \text{for all}\quad (x,\vt)\in \partial Q\times [0,\infty).
\end{align}

In view of Gibb's relation (\ref{m97}),
the internal energy equation \eqref{eq:13} can be rewritten in the form of the entropy balance
\begin{align}\label{eq:13'}
\dd(\varrho s)+\Big[\Div(\varrho s\bfu)+\Div\Big(\frac{\bfq}{\vartheta}\Big)\Big]\dt=\sigma\dt 
\end{align}
with the entropy production rate
\begin{align}\label{eq:0202}
\sigma=\frac{1}{\vartheta}\Big(\bfS:\nabla\bfu-\frac{\bfq\cdot\nabla\vartheta}{\vartheta}\Big).
\end{align}
In view of possible singularities, it is convenient
to relax the equality sign in \eqref{eq:0202} to the inequality
\begin{align}\label{eq:0202bis}
\sigma \geq \frac{1}{\vartheta}\Big(\bfS:\nabla\bfu-\frac{\bfq\cdot\nabla\vartheta}{\vartheta}\Big).
\end{align}

The system is augmented by the total energy balance
\begin{equation} \label{eq0202bisE}
\dd \int_{Q} \left[ \frac{1}{2} \vr |\vu|^2 + \vr e \right] \dx  = \int_{Q} \vr \mathbb{F} \cdot \vu \ \dd W + \sum_{k\geq1}\int_{Q} \frac{1}{2} \vr |\vc{F}_k|^2 \dx\dt-\int_{\partial Q}d(\vt)(\vt-\Theta_0)\dH\dt,
\end{equation}
cf. \cite[Chapter 2]{F}. In case of a stationary solution applying expectations to \eqref{eq0202bisE}
clearly yields
\begin{equation*}
\sum_{k\geq1}\E\int_{Q} \frac{1}{2} \vr |\vc{F}_k|^2 \dx\dt=\E\int_{\partial Q}d(\vt)(\vt-\Theta_0)\dH\dt,
\end{equation*}
meaning energy created by the stochastic forcing can leave through the boundary.
The existence theory from \cite{BF}, which leans on the analysis of the isentropic stochastic Navier--Stokes equations from \cite{BrHo} and the deterministic Navier--Stokes--Fourier equations
from \cite{F}, can be applied to \eqref{eq:1}--(\ref{beq:1}) without essential differences.
In case of the initial value problem an energy estimate can be derived in terms of the initial data.
Looking for stationary solutions, the initial data is not known and one has to use stationarity
instead. In \cite{BFHM} stationarity is used in combination with pressure estimates to obtain
a corresponding estimate for the isentropic problem. When applying the same strategy to the non-isentropic problem \eqref{eq:1}--\eqref{m97},
supplemented with homogeneous boundary conditions for the temperature flux,
the temperature is deemed to grow unboundedly due to the irreversible transfer of the mechanical energy into heat.

Assuming the non-homogeneous boundary conditions (\ref{beq:1}) instead we are able to derive new global-in-time energy estimates, see \eqref{eq:key}. The main task is to control the radiation  energy given by $a\vt^4$ without an information on the initial data. In the case of homogeneous boundary conditions one can only obtain informations on the temperature gradient which is not enough to even get estimates for $\vt$ in $L^1$. For the non-homogeneous problem we benefit from the boundary term in the energy balance \eqref{eq0202bisE}. A suitable application of It\^{o}'s formula combined with Sobolev's embedding and an interpolation argument allows to control a higher power of the temperature in terms of the energy, see \eqref{eq:key}. Finally, we derive some pressure estimate by means of the Bogovskii operator in \eqref{pressure:ep} and \eqref{pressure:delta} to close the argument and to obtain uniform-in-time estimates for the total energy.
This leads to our main result which is the existence of stationary martingale solutions to \eqref{eq:1}--(\ref{beq:1}), see Theorem \ref{Tm1} for the precise statement.

In order to make the ideas just explained rigorous one has to regularise the system
by adding artificial viscosity to the continuity equation \eqref{eq:12} ($\ep$-layer)
and add a high power of the pressure in the momentum equation \eqref{eq:11} ($\delta$-layer).
The resulting system has been solved in \cite{BF} by adding three additional layers.
The same tedious strategy has been applied \cite{BFHM} in the construction of stationary solutions
to the isentropic system. Here, we follow a different strategy with a much simpler proof. Namely, inspired by the approach due to It\^o-Nisio \cite{IN} which we recently also applied to the isentropic system with hard sphere pressure \cite{BFH2020}, we construct stationary solutions directly on the $\ep$-level. 
The first step is to show uniform-in-time estimates for martingale solutions to the initial value problem. In a second step
stationary solutions can be constructed by the Krylov-Bogoliubov method
as the narrow limit of time-averages. A striking feature of this approach is that stationary solutions are sitting on the trajectory space and are approached asymptotically in time by any
solution starting with bounded initial data of certain moments. With a stationary solution to the approximate system at 
hand
one can prove estimate \eqref{eq:key} which is uniform in time, $\varepsilon$
and $\delta$. It has to be combined with pressure estimates which differ on both levels, see \eqref{pressure:ep} and \eqref{pressure:delta}, before one can pass to the limit (both limits have to be done independently). The limit passage
can be performed as in previous papers and stationarity is preserved in the limit.

\section{Mathematical framework and the main result}
\label{sec:framework}

\subsection{Stochastic forcing}

The process $W$ is a cylindrical Wiener process on a separable Hilbert space $\mathfrak U$, that is, $W(t)=\sum_{k\geq1}\beta_k(t) e_k$ with $(\beta_k)_{k\geq1}$ being mutually independent real-valued standard Wiener processes relative to $(\mf_t)_{t\geq0}$. Here $(e_k)_{k\geq1}$ denotes a complete orthonormal system in $\mathfrak{U}$. In addition, we introduce an auxiliary space $\mathfrak{U}_0\supset\mathfrak{U}$ via
$$\mathfrak{U}_0=\bigg\{v=\sum_{k\geq1}\alpha_k e_k;\;\sum_{k\geq1}\frac{\alpha_k^2}{k^2}<\infty\bigg\},$$
endowed with the norm
$$\|v\|^2_{\mathfrak{U}_0}=\sum_{k\geq1}\frac{\alpha_k^2}{k^2},\qquad v=\sum_{k\geq1}\alpha_k e_k.$$
Note that the embedding $\mathfrak{U}\hookrightarrow\mathfrak{U}_0$ is Hilbert-Schmidt. Moreover, trajectories of $W$ are $\prst$-a.s. in $C([0,T];\mathfrak{U}_0)$ (see \cite{daprato}).

Choosing $\mathfrak{U}= \ell^2$ we may identify the diffusion coefficients
$(\mathbb {F} e_k )_{k \geq 1}$ with a sequence of real functions
$(\vc{F}_k)_{k \geq 1}$,
\[
\vr \mathbb{F}(\vr, \vt, \vu) \dd W = \sum_{k=1}^\infty \vr \vc{F}_k(x, \vr, \vt, \vu) \dd \beta_k.
\]
\\
We suppose that $\vc{F}_k$ are smooth in their arguments, specifically,
\[
\vc{F}_k \in C^1( \Ov{Q} \times [0, \infty)^2 \times R^3; R^3),
\]
where 
\begin{equation} \label{P-1}
\| {\vc{F}_k} \|_{L^\infty} + \| {\nabla_{x,\vr,\vt, \vc{u}}} \vc{F}_k \|_{L^\infty} \leq f_k,\ \sum_{k=1}^\infty f_k^2 < \infty.
\end{equation}
We easily deduce from \eqref{P-1} the following bound
\[
\| \vr \vc{F}_k (\vr, \vt, \vu) \|_{W^{-k,2}(Q; R^3)} \aleq \| \vr \vc{F}_k (\vr, \vt, \vu)  \|_{L^1 (Q; R^3)} \aleq f_k
 \| \vr \|_{L^1(Q)} 
\]
whenever $k > \frac{3}{2}$.
Accordingly, the stochastic integral
\[
\int_0^\tau \vr \mathbb{F} \dd W = \sum_{k = 1}^\infty \int_0^\tau \vr \vc{F}_k(\vr, \vt, \vu) \ \dd \beta_k
\]
can be identified with an element of the {Banach space space $C([0,T]; W^{-k,2}(Q))$},
\[
\begin{split}
&\intQ{ \left( \int_0^\tau \vr \mathbb{F}(\vr, \vt, \vu) \dd W \cdot \bfvarphi \right) } \\ &= \sum_{k=1}^\infty \int_0^\tau \left(
\intQ{ \vr \vc{F}_k(x,\vr, \vt, \vu) \cdot \varphi } \right) \dd \beta_k, \ \bfvarphi \in W^{k,2}(Q; R^3),\ k > \frac{3}{2}.
\end{split}
\]

\subsection{Structural and constitutive assumptions}
\label{H}

Besides Gibbs' equation (\ref{m97}), we impose several restrictions on the specific shape of the thermodynamic functions
$p=p(\vr,\vt)$, $e=e(\vr,\vt)$ and $s=s(\vr,\vt)$. They are borrowed from \cite[Chapter 1]{F}, to which we refer
for the physical background and the relevant discussion.

We consider the pressure $p$ in the form
\bFormula{m98}
p(\vr, \vt) = p_M(\vr, \vt) + \frac{a}{3} \vt^4, \ a > 0,\
p_M(\vr, \vt) = \vt^{5/2} P\left( \frac{\vr}{\vt^{3/2}} \right),
\eF
\begin{equation} \label{mp8a}
e(\vr, \vt) = e_M(\vr, \vt) + {a} \frac{\vt^4}{\vr},\ e_M(\vr, \vt) = \frac{3}{2} \frac{p_M(\vr, \vt)}{\vr} =
\frac{3}{2} \frac{\vt^{5/2}}{\vr} P\left( \frac{\vr}{\vt^{3/2}} \right),
\end{equation}
\begin{equation}\label{md8!}
s(\vr, \vt) = s_M(\vr, \vt) + \frac{4a }{3} \frac{\vt^3}{\vr},\
s_M(\vr, \vt) = S \left( \frac{\vr} {\vt^{3/2}} \right),
\end{equation}
\begin{equation}\label{md8}
S = S(Z),\
S'(Z)=-\frac 32 \frac{\frac 53 P(Z)-ZP'(Z)}{Z^2}<0,\ \lim_{Z\to\infty} S(Z)=0,
\end{equation}
where
\begin{equation}\label{m103-}
P \in C^1[0, \infty) \cap C^2(0, \infty),\; P(0)= 0,\;P'(Z)>0, \ \mbox{for all}\ {Z \geq 0},
\end{equation}
\begin{equation}\label{md7}
0<\frac 32 \frac{\frac 53 P(Z)-ZP'(Z)}{Z}  <c,\; \mbox{for all}\ Z > 0,
\end{equation}
and
\begin{equation}\label{md1}
\lim_{Z\to\infty}\frac {P(Z)} {Z^{5/3}}=p_\infty>0.
\end{equation}

As shown in \cite[Section 3.2]{F} the assumptions above imply that there is $c>0$ such that
\begin{align}\label{3.32}
& c^{-1} \vr^{5/3}\leq p_M(\vr,\vt)\leq \,c(\vr^{5/3}+\vr\vt),\\
&\frac{3p_\infty}{2}\vr^{5/3}+ a\vt^4\leq \vr e(\vr,\vt),\label{3.28}\\
& 0\leq e_M(\vr,\vt)\leq \,\overline c(\vr^{2/3}+\vt),\label{3.30}
\end{align}
for all $\vt,\vr>0$. Moreover, there is $s_\infty>0$ such that
\begin{align}
&0\leq s_M(\vr,\vt)\leq\,s_\infty (1+|\log(\vr)|+ [\log(\vt)]^+ ).\label{3.39}
\end{align}
Finally, for $\overline\vt>0$ we  introduce \emph{ballistic free energy} given by
\[
H_{\overline{\vt}}(\vr, \vt) = \vr \left( e(\vr, \vt) - \overline{\vt} s (\vr, \vt) \right),
\]
which satisfies
\begin{align}\label{est:H}
-c(\vr+1)+\frac{1}{4}\big(\vr e(\vr,\vt)+\overline\vt|s(\vr,\vt)|\big)\leq H_{\overline{\vt}}(\vr, \vt)\leq\,c\big(\vr^{5/3}+\vt^4+1\big)
\end{align}
on account of \eqref{3.30}, \eqref{3.39} and \cite[Prop. 3.2]{F}.
The viscosity coefficients $\mu$, $\eta$ are continuously
differentiable functions of the absolute temperature $\vt$, more
precisely $\mu, \ \lambda  \in C^1[0,\infty)$, satisfying
\bFormula{m105} 0 < \underline{\mu}(1 + \vartheta) \leq
\mu(\vartheta) \leq \Ov{\mu}(1 + \vartheta), \eF
\bFormula{*m105*} \sup_{\vartheta\in [0, \infty)}\big(|\mu'(\vartheta)|+|\lambda'(\vartheta)|\big)\le
\overline m, \eF
\bFormula{m106} 0 \leq \lambda (\vartheta) \leq \Ov{\lambda}(1 +
\vartheta). \eF
\\
The heat conductivity coefficient $\kappa \in C^1[0, \infty)$ satisfies
\bFormula{m108} 0 < \underline{\kappa} (1 + \vartheta^3) \leq \kappa(
\vartheta) \leq \Ov{\kappa} (1 + \vartheta^3).
\eF

Finally, we introduce certain regularised versions of $p,e,s$ and $\kappa$ for fixed $\delta>0$:
\begin{equation} \label{neu}
\begin{split}
p_\delta(\vr,\vt)&=p (\vr,\vt) + \delta({\vr^2} + \vr^\Gamma), \\
e_{M, \delta}(\vr, \vt) &= e_M(\vr, \vt) + \delta \vt,\
e_\delta(\vr, \vt) = e(\vr, \vt) + \delta \vt, \\
s_{M,\delta}(\vartheta,\varrho)  &=s_M(\vartheta,\varrho)+\delta\log\vartheta,\quad s_\delta(\vr, \vt) = s(\vr, \vt) + \delta \log(\vt),\\ \kappa_\delta(\vartheta )&=\kappa(\vartheta )+\delta\Big(\vt^\Gamma+\frac{1}{\vartheta}\Big),\ \cal{K}_\delta(\vt) = \int_0^\vt \kappa_\delta(z) \ {\rm d}z.
\end{split}
\end{equation}

\subsection{Martingale \& stationary solutions}
\label{subsec:solution}

We start with a rigorous definition of \emph{(weak) martingale solution} to problem \eqref{eq:1}--(\ref{beq:1}) as given in \cite{BF}, where also the existence of a solution to the corresponding initial value problem is proved.

\begin{definition}[Martingale solution]\label{def:sol}
Let $Q \subset R^3$ be a bounded domain of class $C^{2 + \nu}$, $\nu > 0$.
Then
$$\big((\Omega,\mf,(\mf_t),\prst),\varrho,\vartheta,\bfu,W)$$
is called \emph{(weak) martingale solution} to problem \eqref{eq:1}--(\ref{beq:1}) provided the following holds.
\begin{enumerate}
\item $(\Omega,\mf,(\mf_t),\prst)$ is a stochastic basis with a complete right-continuous filtration;
\item $W$ is an $(\mf_t)$-cylindrical Wiener process;
\item the random variables
\[
\vr \in L^1_{\rm loc}([0,\infty);L^1(Q)),\ \vt \in L^1_{\rm loc}([0,\infty);L^1(Q)),\ \vu \in L^2_{\rm loc}([0,\infty); W^{1,2}_0(Q; R^3))
\]
are $(\mf_t)$-progressively measurable\footnote{The progressive measurability is understood in the sense of random distributions as introduced in \cite[Section 2.2]{BFHbook}.}, $\vr \geq 0$, $\vt > 0$ $\pas$;
\item the equation of continuity
\begin{align}\label{wWS237final}
\int_0^\infty \intQ{ \left[ \vr \partial_t \psi + \vr \vu \cdot \Grad \psi \right] }\dt = 0;
\end{align}
holds for all $\psi\in \DC((0,\infty) \times R^3)$ $\prst$-a.s.;
\item the momentum equation
\begin{align}\label{wWS238final}
\begin{aligned}
&\int_0^\infty \partial_t \psi \intQ{ \vr \vu \cdot \bfvarphi } \dt\\
&+\int_0^\infty \psi \intQ{ \varrho\bfu\otimes\bfu : \nabla \bfvarphi } \dt -\int_0^T \psi \intQ{ \mathbb S(\vartheta,\nabla\bfu) : \nabla\bfvarphi } \dt\\
&+\int_0^\infty \psi \intQ{ p(\varrho,\vartheta) \diver\bfvarphi }\dt + \int_0^\infty \psi \intQ{ \varrho{\vc{F}}(\varrho, \vt , \bfu) \cdot
\bfvarphi }\,\dif W =0;
\end{aligned}
\end{align}
holds for all $\psi \in \DC(0,\infty)$,
 $\bfvarphi\in \DC (Q;R^3)$ $\prst$-a.s.
\item the entropy balance
\begin{align} \label{m217*final}\begin{aligned}
-
 \int_0^\infty &\intQ{ \left[\varrho s(\varrho,\vartheta) \partial_t \psi + \varrho s (\varrho,\vartheta)\bfu \cdot \nabla\psi \right]}\dif t\\\geq & \int_0^\infty\int_{Q}
\frac{1}{\vartheta}\Big[\mathbb S(\vartheta,\nabla\bfu):\nabla\bfu+\frac{\kappa(\vartheta)}{\vartheta}|\nabla\vartheta|^2\Big]
\psi\,\dif x\,\dif t \\
&- \int_0^\infty\intQ{ \frac{\kappa(\vartheta)\nabla\vartheta}{\vartheta} \cdot \nabla\psi }\,\dif t
-\int_0^\infty\int_{\partial Q}\psi\frac{d(\vt)}{\vt}(\vt-\Theta_0)\dH\dt
\end{aligned}
\end{align}
holds for all
$\psi\in \DC((0,\infty)\times R^3)$, $\psi \geq 0$ $\prst$-a.s.;
\item the total energy balance
\begin{equation} \label{EI20final}
\begin{split}
- \int_0^\infty &\partial_t \psi 
\left( \intQ{ \mathcal E(\varrho,\vartheta, \vu) } \right) \ \dt = -\int_0^\infty\psi\int_{\partial Q}d(\vt-\Theta_0)\dH\dt\\&+\int_0^\infty  \psi \int_{Q}\varrho{\mathbb F}(\varrho,\vartheta,\bfu)\cdot\bfu\dx\,\dd W\dx + \frac{1}{2} \int_0^\infty
\psi \bigg(\intQ{ \sum_{k \geq 1}  \varrho| {\bf F}_k (\varrho,\vartheta,  {\bf u}) |^2 } \bigg) {\rm d}t
\end{split}
\end{equation}
holds for any $\psi \in \DC(0, \infty)$ $\mathbb{P}$-a.s.
Here, we abbreviated
$$\mathcal E(\varrho,\vartheta,\vu)= \frac{1}{2} \varrho | {\bf u} |^2 + \varrho e(\varrho,\vartheta).$$
\end{enumerate}
\end{definition}

In the following we are going to introduce the concept of stationary martingale solutions.
We start with a standard definition of stationarity for stochastic processes with values in Sobolev spaces.

\begin{Definition}[Classical stationarity]\label{D2}
Let $\bfU=\{\bfU(t);t\in[0,\infty)\}$ be an $W^{k,p}(\mt)$-valued measurable stochastic process, where $k\in\mn_0$ and $p\in[1,\infty)$. We say that $\bfU$ is \emph{stationary} on $W^{k,p}(\mt)$ provided the joint laws
$$\mathcal{L}(\bfU(t_1+\tau),\dots, \bfU(t_n+\tau)),\quad \mathcal{L}(\bfU(t_1),\dots, \bfU(t_n))$$
on $[W^{k,p}(\mt)]^n$ coincide for all $\tau\geq0$, for all $t_1,\dots,t_n\in [0,\infty)$.

\end{Definition}
As can be seen from Definition \ref{def:sol}, the velocity $\vu$ and the temperature $\vt$ are not stationary in the sense of Definition \ref{D2} as they are only equivalence classes in time.
Therefore we use the following definition of stationarity which has been introduced in \cite{BFHM},
and applies to random variables ranging in the
space $L^q_{\rm loc}([0, \infty);W^{k,p}(\mt))$.

\begin{Definition}[Weak stationarity]\label{D1}
Let $\bfU$ be an $L^q_{\rm loc}([0,\infty);W^{k,p}(\mt))$-valued random variable, where Let $k\in\mn_0$ and $p,q\in[1,\infty)$. Let $\mathcal S_\tau$ be the time shift on the space of trajectories given by $\mathcal{S}_\tau \bfU(t)=\bfU(t+\tau).$
We say that $\bfU$ is \emph{stationary} on $L^q_{\rm loc}([0,\infty);W^{k,p}(\mt))$ provided the laws
$\mathcal{L}(\mathcal{S}_\tau\bfU),$ $ \mathcal{L}(\bfU)$
on $L^q_{\rm loc}([0,\infty);W^{k,p}(\mt))$ coincide for all $\tau\geq0$.

\end{Definition} 
Definition \ref{D2} and Definition \ref{D1} are equivalent as soon as the stochastic process in question is
continuous in time; or alternatively, if it is weakly continuous and satisfies a suitable uniform bound, cf. \cite[Lemma A.2 and Corollary A.3]{BFHM}. Furthermore, it can be shown that both notions of stationarity are stable under weak convergence as can be seen from the following two lemmas (the proofs of which can be found in \cite[Appendix]{BFHM}).
\begin{Lemma}\label{lem:stac}
Let $k\in\mn_0, p,q\in[1,\infty)$ and let $(\bfU_m)$ be a sequence of random variables taking values in $L^q_{\rm loc}([0,\infty);W^{k,p}(\mt)))$. If, for all $m\in\mn$, $\bfU_m$ is stationary on $L^q_{\rm loc}([0,\infty);W^{k,p}(\mt))$ in the sense of Definition \ref{D1} and
\begin{align*}
\bfU_m\rightharpoonup \bfU\quad\text{in}\quad L^q_{\mathrm{loc}}([0,\infty);W^{k,p}(\mt))\quad\mathbb P\text{-a.s.,}
\end{align*}
then $\bfU$ is stationary on $L^q_{\rm loc}([0,\infty);W^{k,p}(\mt))$.
\end{Lemma}
\begin{Lemma}\label{lem:stac2}
Let $k\in\mn_0$, $p\in[1,\infty)$ and let $(\bfU_m)$ be a sequence of $W^{k,p}(\mt)$-valued  stochastic processes which are stationary on $W^{k,p}(\mt)$ in the sense of Definition \ref{D2}. If for all $T>0$
\begin{equation}\label{eq:gh1}
\sup_{m\in\mn}\expe{\sup_{t\in[0,T]}\|\bfU_m\|_{W^{k,p}(\mt)}}<\infty
\end{equation}
and
\begin{align*}
\bfU_m\rightarrow \bfU\quad\text{in}\quad C_{\mathrm{loc}}([0,\infty);(W^{k,p}(\mt),w))\quad\mathbb P\text{-a.s.,}
\end{align*}
then $\bfU$ is stationary on $W^{k,p}(\mt)$.
\end{Lemma}
In the following we define a stationary martingale solution to \eqref{eq:1}--(\ref{beq:1}).

\begin{Definition}\label{Dm2}
A weak martingale solution $[\vr,\vt,\vu, W]$ to \eqref{eq:1}--\eqref{beq:1} is called {\em stationary} provided 
the joint law of the time shift 
$
\left[\mathcal{S}_\tau\vr,\mathcal{S}_\tau\vt,  \mathcal{S}_\tau\vu, \mathcal{S}_\tau W - W (\tau)\right]$ on
$$ L^1_{\rm loc}([0, \infty); L^\gamma(\Q))\times L^1_{\rm loc}([0, \infty); W^{1,2}(Q)) \times L^1_{\rm loc}([0, \infty); W^{1,2}(Q; \R^3)) 
\times C([0, \infty); \mathfrak{U}_0 )
$$
is independent of $\tau \geq 0$.
\end{Definition}
We now state our main result concerning the existence of a stationary martingale solution to \eqref{eq:1}--(\ref{beq:1}).
\begin{Theorem} \label{Tm1}
Let $M_0 \in(0,\infty)$ be  given. Suppose that the structural assumptions \eqref{m98}--\eqref{m108} are in force and that the diffusion coefficient satisfies
$\mathbb F$ satisfies \eqref{P-1}.
Then problem \eqref{eq:1}--\eqref{beq:1} admits a {stationary martingale solution} in the sense of Definition \ref{Dm2}.
\end{Theorem}
The proof of Theorem \ref{Tm1} is split into several parts. In the next section
we study the approximate system with regularisation parameters $\varepsilon$ and $\delta$. The proof will be completed
in Section \ref{sec:lim} after passing to the limit in $\ep$ and $\delta$.

\section{The viscous approximation}
\label{sec:vis}

In this section we study the viscous approximation to \eqref{eq:1}--\eqref{beq:1}, where the continuity equation contains an artificial diffusion ($\ep$-layer) and the pressure is stabilised by an artificial high power to the density ($\delta$-layer). In addition to the common
terms we add additional stabilising quantities in the continuity equations as in \cite{BFHM}, see \eqref{L1} below.

\subsection{Martingale solutions}
In this subsection we give a precise formulation of the approximated problem.
For this purpose we introduce a cut-off function
\[
\chi \in C^\infty(R), \ \chi(z) = \left\{ \begin{array}{l} 1 \ \mbox{for} \ z \leq 0, \\ \chi'(z) \leq 0 \ \mbox{for}\ 0 < z < 1, \\
\chi(z) = 0 \ \mbox{for}\ z \geq 1. \end{array} \right.
\]
We denote by $M_\ep$ the unique solution to the equation $2\ep z=\chi(z/M_0)$ which obviously satisfies $M_\ep\leq M_0$.
Finally, the diffusion coefficients are regularised by replacing $\vc{F}$ by $\vc{F}_\ep$,
\[
\mathbb{F}_{\ep} = \left( \vc{F}_{k,\ep} \right)_{k \geq 1},\
\vc{F}_{k,\ep}(x,\vr, \vt, \vu) = \chi \left( \frac{\ep}{\vr} - 1 \right) \chi \left(|\bfu|- \frac{1}{\ep} \right)
\vc{F}_k(x, \vr, \vt, \vu).
\]
 Let us start with a precise formulation of the problem.

\begin{itemize}
\item {\bf Regularized equation of continuity.}
\begin{equation} \label{L1}
\begin{split}
\int_{0}^{\infty} &\intO{ \left[ \vr \partial_t \varphi + \vr \vu \cdot \Grad \varphi \right] } \dt \\ &= \ep \int_0^\infty \intO{ \left[\Grad \vr \cdot \Grad \varphi
- 2 \vr \varphi \right] }\dt - 2\ep \int_0^\infty \intO{ M_{\ep} \varphi } \dt
\end{split}
\end{equation}
for any $\varphi \in \DC((0, \infty) \times \mt)$ $\prst$-a.s. 
\item {\bf Regularized momentum equation.}
\begin{equation} \label{L2}
\begin{split}
\int_0^\infty \partial_t \psi&\int_Q  \vr \vu \cdot \varphi\dx \dt  + \int_0^\infty\psi  \intO{ \vr \vu \otimes \vu : \Grad \varphi } \dt
+ \int_0^\infty \psi \intO{ p_\delta(\vt,\vr)  \Div \varphi } \dt  \\
&-   \int_0^\infty\psi  \intO{ \mathbb{S}_\delta(\vt,\Grad \vu) : \Grad \varphi} \dt  - \ep \int_0^\infty \psi \intO{  \vr \vu \cdot \Delta \varphi } \dt \\&- 2 \ep \int_0^\infty \psi \intO{ \vr \vu \cdot \varphi } \ \dt
=- \int_0^\infty \psi \intO{ \vr\mathbb{F}_\ep(\vr,\vt,\bfu) \cdot \varphi } \ \D W
\end{split}
\end{equation}
for any $\psi \in \DC((0, \infty))$, $\varphi \in C^\infty(\mt; \R^3)$ $\prst$-a.s.
\item {\bf Regularized entropy balance.}
\begin{equation} \label{L3}
\begin{aligned}
-
 \int_0^\infty &\intQ{ \left[\varrho s_\delta(\varrho,\vartheta) \partial_t \psi + \varrho s_\delta (\varrho,\vartheta)\bfu \cdot \nabla\psi \right]\,\varphi}\dif t\\\geq & \int_0^\infty\int_{Q}
\frac{1}{\vartheta}\Big[\mathbb S(\vartheta,\nabla\bfu):\nabla\bfu+\frac{\kappa_\delta(\vartheta)}{\vartheta}|\nabla\vartheta|^2+\delta\frac{1}{\vt^2}\Big]
\psi\,\dif x\,\dif t \\
&+ \int_0^\infty\psi\intQ{ \frac{\kappa_\delta(\vartheta)\nabla\vartheta}{\vartheta} \cdot \nabla\varphi }\,\dif t-\int_0^\infty\psi\int_{\partial Q}\varphi\frac{d(\vt)}{\vt}(\vt-\Theta_0)\dH\dt\\
&-\ep \int_0^\infty\psi\int_Q\left[ \left( \vt s_{M,\delta} (\vr, \vt) - e_{M, \delta} (\vr, \vt) - \frac{p_M (\vr, \vt) }{\vr} \right) \frac{ \Grad \vr}{\vt} \right]\cdot\nabla\varphi\dxt\\
&+\int_0^\infty\psi\int_Q\left[ \frac{\ep \delta}{2 \vartheta} ( \beta \varrho^{\beta - 2} + 2)
|\Grad \varrho|^2+\ep \frac{1}{\vr \vt} \frac{\partial p_M}{\partial \vr} (\vr, \vt) |\Grad \vr|^2 - \ep \vartheta^4\right]\varphi \dxt\\
&+\int_0^\infty\psi\int_Q\big(-2\varepsilon \varrho +2\ep M_\ep\big)\frac{1}{\vt}\Big( \vartheta
s_{M,\delta}(\varrho, \vartheta) - e_{M,\delta}(\varrho, \vartheta) -
\frac{p_M(\varrho, \vartheta)}{\varrho} \Big)\,\varphi\dxt
\end{aligned}
\end{equation}
for any $\psi \in \DC((0, \infty))$, $\varphi \in C^\infty(\mt; \R^3)$ $\prst$-a.s.
\item {\bf Regularized total energy balance.}
\begin{equation} \label{L4}
\begin{split}
- \int_0^\infty& \partial_t \psi
\bigg(\int_{Q} {\mathcal E_\delta(\varrho,\vartheta, \vu) } \dx\bigg) \dt
 + \int_0^\infty {\psi} \int_{Q} \ep \vartheta^5 \dt \\&+2 \ep\int_0^T\psi \intO{ \left[  \delta\vr^2 + \frac{\delta \Gamma }{\Gamma - 1} \vr^{\Gamma}+\frac{1}{2}\vr|\bfu|^2 \right] } \dt+\int_0^\infty\psi\int_{\partial Q}d(\vt-\Theta_0)\dH\dt \\
&= \,\int_0^\infty\int_{Q}\frac{\delta}{\vartheta^2}\psi\dx\dt+\int_0^\infty\psi \ep M_\ep \intO{ \left( 2\delta \vr + \frac{\delta \Gamma}{\Gamma - 1} \vr^{\Gamma - 1}
 +\frac{1}{2}|\bfu|^2\right) }  \dt\\
&+ \frac{1}{2} \int_0^T
\psi \bigg(
\intQ{ \sum_{k \geq 1} \varrho | {\bf F}_{k,\ep} (\varrho,\vartheta, {\bf u}) |^2  } \bigg) {\rm d}t+\int_0^\infty  \psi\int_{Q}\varrho{\vc{F}_\ep}(\varrho,\vartheta,\bfu)\cdot\bfu\,\dd W\dx
\end{split}
\end{equation}
holds for any $\psi \in \DC(0, \infty)$ $\mathbb{P}$-a.s., where we have set
\begin{align*}
\mathcal E_\delta=\frac{1}{2}\vr|\bfu|^2+\vr e_\delta(\vr,\vt)+\delta\Big(\vr^2+\frac{1}{\Gamma-1}\vr^\Gamma\Big).
\end{align*}
\end{itemize}

we have the following result.

\begin{Proposition}\label{prop:1608b}
Let $\varepsilon,\delta>0$ be given.
Then there exists a weak martingale solution $[\varrho_\varepsilon,\vte,\bfu_\varepsilon]$ to \eqref{L1}--\eqref{L4}.
In addition, for $n\in\mn$ and  every  $\psi\in C^\infty_c((0,\infty))$, $\psi\geq 0$, the following generalized energy inequality holds true
 \begin{align} \label{b2a}
\begin{aligned}
-&\int_0^\infty\partial_t\psi \Big[  \intQ{ E_H^{\delta,\overline{\vt}}(\varrho,\vartheta)  }\Big]^n\dt\\
&+n\overline{\vt} \int_{0}^{\infty}\psi\Big[  \intQ{ E_H^{\delta,\overline{\vt}}(\varrho,\vartheta)  (\tau, \cdot) }\Big]^{n-1}\int_{\mt}\sigma_{\varepsilon,\delta}\dx\dt\\
&+ n\int_{0}^{\infty}\psi\Big[  \intQ{ E_H^{\delta,\overline{\vt}}(\varrho,\vartheta)  (\tau, \cdot) }\Big]^{n-1}\int_{Q}\ep \vartheta^5  \dt\\
&+n\int_{0}^{\infty}\psi\Big[  \intQ{ E_H^{\delta,\overline{\vt}}(\varrho,\vartheta)  (\tau, \cdot) }\Big]^{n-1}\int_{\partial Q}\frac{d(\vt)}{\vt}\Big(\vt-\overline{\vt}\Big)(\vt-\Theta_0)\dH\dt\\
&+ 2\ep n\overline\vt\E\int_{0}^{\infty}\psi\Big[  \intQ{ E_H^{\delta,\overline{\vt}}(\varrho,\vartheta)  (\tau, \cdot) }\Big]^{n-1}\intO{ \left[  \delta\vr^2 + \frac{\delta \Gamma }{\Gamma - 1} \vr^{\Gamma} +\frac{1}{2}\vr|\bfu|^2\right] }  \dt\\
&+2\varepsilon n \int_{0}^{\infty}\psi \Big[  \intQ{ E_H^{\delta,\overline{\vt}}(\varrho,\vartheta)  (\tau, \cdot) }\Big]^{n-1}\int_{Q}\varrho\Big( \frac{p_M(\varrho, \vartheta)}{\varrho\vt}
+\frac{ e_{M,\delta}(\varrho, \vartheta)}{\vt} -s_{M,\delta}(\varrho, \vartheta) 
 \Big)\dxt\\
&+n\varepsilon\int_{0}^{\infty}\psi\Big[  \intQ{ E_H^{\delta,\overline{\vt}}(\varrho,\vartheta)  (\tau, \cdot) }\Big]^{n-1}\int_{Q}\frac{\overline{\vt}}{\vartheta^2}\bigg(e_{M,\delta}(\varrho,\vartheta)+\varrho\frac{\partial e_M}{\partial\varrho}(\varrho,\vartheta)\bigg)\nabla\varrho\cdot\nabla\vartheta\dx\dt \ \\
&+n\int_{0}^{\infty}\psi\Big[  \intQ{ E_H^{\delta,\overline{\vt}}(\varrho,\vartheta)  (\tau, \cdot) }\Big]^{n-1} \int_{Q}\varrho{\vc{F}}_\ep (\varrho,\vartheta,\bfu)\cdot\bfu\,\dd W\dx\\&+ \frac{n}{2}\int_{0}^{\infty}\psi
\Big[  \intQ{ E_H^{\delta,\overline{\vt}}(\varrho,\vartheta)  (\tau, \cdot) }\Big]^{n-1}\bigg(
\intQ{ \sum_{k \geq 1} \varrho | {\bf F}_{k,\ep} (\varrho,\vartheta,{\bf u}) |^2  } \bigg) {\rm d}t\\
&+ n\ep M_\ep\int_{0}^{\infty}\psi \Big[  \intQ{ E_H^{\delta,\overline{\vt}}(\varrho,\vartheta)  (\tau, \cdot) }\Big]^{n-1} \intO{ \left( 2\delta \vr + \frac{\delta \Gamma}{\Gamma - 1} \vr^{\Gamma - 1}+\frac{1}{2}|\bfu|^2\right) }  \dt\\
&+ n\overline\vt\ep M_\ep\int_{0}^{\infty}\psi \Big[  \intQ{ E_H^{\delta,\overline{\vt}}(\varrho,\vartheta)  (\tau, \cdot) }\Big]^{n-1} \intO{ \left( \frac{p_M(\varrho, \vartheta)}{\varrho\vt}+\frac{e_{M,\delta}(\varrho, \vartheta)}{\vt}
-
s_{M,\delta}(\varrho, \vartheta) \right) }  \dt\\
&+\frac{n(n-1)}{2}\int_{0}^{\infty}\psi \Big[  \intQ{ E_H^{\delta,\overline{\vt}}(\varrho,\vartheta)  (\tau, \cdot) }\Big]^{n-2}\sum_{k=1}^\infty\int_{\tau_1}^{\tau_2}\bigg( \int_{Q}\varrho\vc F_{k,\ep} (\varrho,\vartheta,\bfu)\cdot\bfu\dx\bigg)^2\dt.
\end{aligned}
\end{align}
Here we abbreviated
\begin{align*}
\sigma_{\varepsilon,\delta}&=\frac{1}{\vartheta}\Big[\mathbb S(\vartheta,\nabla\bfu):\nabla\bfu+\frac{\kappa(\vartheta)}{\vartheta}|\nabla\vartheta|^2+\frac{\delta}{2}\Big(\vt^{\Gamma-1}+\frac{1}{\vartheta^2}\Big)|\nabla\vartheta|^2+\delta\frac{1}{\vartheta^2}\Big]\\
&+\frac{\varepsilon\delta}{2\vartheta} {\left( \beta\varrho^{\Gamma-2} + 2 \right)} |\nabla\varrho|^2+\varepsilon\frac{\partial p_M}{\partial\varrho}(\varrho,\vartheta)\frac{|\nabla\varrho|^2}{\varrho\vartheta} + \varepsilon\frac{\vr}{\vt} |\Grad \vu|^2 ,
\end{align*}
and $E_{H}^{\delta, \overline{\vt}}=\frac{1}{2}\vr|\bfu|^2+H_{\overline{\vt}}(\vr, \vt)+\delta\big(\vr^2+\frac{1}{\Gamma-1}\vr^\Gamma\big)$,
where
\begin{align}\label{eq:Hdelta}
H_{\delta, \overline{\vt}}(\vr, \vt) = \vr \left( e_\delta(\vr, \vt) - \overline{\vt} s_\delta (\vr, \vt) \right)=H_{\overline\vt}(\vr, \vt)+\delta\vr\vt-\overline\vt \vr\log(\vt)
\end{align}
with $H_{\overline{\vt}}(\vr, \vt)$ introduced in \eqref{H}.
\end{Proposition}
\begin{proof}
Although there are some differences to system (4.24)--(4.27) from \cite{BF} the method
still applies (in particular, it is possible to allow an unbounded time interval by working with spaces of the from $L^q_{\rm loc}([0,\infty);X)$ and $C_{\rm loc}([,0,\infty);X)$ for Banach spaces $X$) and we obtain the existence of a weak martingale solution to \eqref{L1}--\eqref{L4}.
We remark, in particular, that the solution in \cite{BF} is constructed with respect to some initial law which does not play any role in our analysis. For simplicity we choose
\begin{align*}
\vr_0=1,\quad\vt_0=1,\quad\bfu_0=0,
\end{align*}
which satisfies all the required assumptions.

 As far as the energy inequality is concerned, the required version can be derived on the basic approximate level (even with equality) and it is preserved in the limit. 
 In fact, one can argue as in \cite[Section 4.1]{BF} to derive the version for $n=1$, while the case $n\geq2$ follows easily from It\^{o}'s formula. It is worth to point out that this procedure
 to test the continuity equation \eqref{L1} with 
 $\frac{1}{2}|\bfu|^2$ and $2\delta\vr+\frac{\delta\Gamma}{\Gamma-1}\vr^\Gamma$ gives rise
 to the terms
$$ 2\ep n\overline\vt\E\int_{0}^{\infty}\psi\Big[  \intQ{ E_H^{\delta,\overline{\vt}}(\varrho,\vartheta)  (\tau, \cdot) }\Big]^{n-1}\intO{ \left[  \delta\vr^2 + \frac{\delta \Gamma }{\Gamma - 1} \vr^{\Gamma} +\frac{1}{2}\vr|\bfu|^2\right] }  \dt$$
 and
 $$n\ep M_\ep\int_{0}^{\infty}\psi \Big[  \intQ{ E_H^{\delta,\overline{\vt}}(\varrho,\vartheta)  (\tau, \cdot) }\Big]^{n-1} \intO{ \left( 2\delta \vr + \frac{\delta \Gamma}{\Gamma - 1} \vr^{\Gamma - 1}+\frac{1}{2}|\bfu|^2\right) }  \dt$$
in \eqref{b2a} which are new in comparison to 
 \cite{BF}. Also, the term
$$n\overline\vt\ep M_\ep\int_{0}^{\infty}\psi \Big[  \intQ{ E_H^{\delta,\overline{\vt}}(\varrho,\vartheta)  (\tau, \cdot) }\Big]^{n-1} \intO{ \left( \frac{p_M(\varrho, \vartheta)}{\varrho\vt}+\frac{e_{M,\delta}(\varrho, \vartheta)}{\vt}
-
s_{M,\delta}(\varrho, \vartheta) \right) }  \dt,$$
which arises due to the last line in \eqref{L3}, does not appear in \cite{BF}. Finally, as in \eqref{eq0202bisE} we have the boundary term due to non-homogeneous boundary conditions
being incorporated already in \eqref{L3}.
\end{proof}

\subsection{Uniform-in-time estimates}
The first step is now to derive estimates which are uniform in time.
\begin{Proposition}\label{prop:en}
Let $(\varrho,\vt,\bfu)$ be a weak martingale solution to \eqref{L1}--\eqref{L4}.
Assume that
\begin{equation} \label{hypo1a}
{\rm ess} \limsup_{t \to 0+} \E\Big[  \intQ{ E_H^{\delta,\overline{\vt}}(\varrho,\vartheta)  (t, \cdot) }\Big]^n<\infty
\end{equation}
for some
$n\in\N$. Then for any $\overline\vt>0$ and $\ep\leq \ep_0$ there is $E_\infty=E_\infty(n,\ep,\delta,\vt)$ such that
\begin{align}\label{eq:2711}
\E\Big[  \intQ{ E_H^{\delta,\overline{\vt}}(\varrho,\vartheta)  (\tau, \cdot) }\Big]^n&\leq E_\infty,\end{align}
as well as
\begin{align}\label{eq:2711a}
\E\int_{0}^{\tau}\Big[  \intQ{ E_H^{\delta,\overline{\vt}}(\varrho,\vartheta)  (\tau, \cdot) }\Big]^{n-1}\bigg(\int_{\mt}\sigma_{\varepsilon,\delta}\dx+\int_Q\ep \vartheta^5\dx\bigg) \dt
&\leq\,E_\infty(1+\tau)
\end{align}
for any $\tau > 0$.
\end{Proposition}
\begin{proof}
The energy inequality in \eqref{b2a} yields for any $0\leq \tau_1<\tau_2$
 \begin{align} \label{wWS24'}
\begin{aligned}
&
 \Big[  \intQ{ E_H^{\delta,\overline{\vt}}(\varrho,\vartheta)  (\tau_2, \cdot) }\Big]^n
+n\overline{\vt} \int_{\tau_1}^{\tau_2}\Big[  \intQ{ E_H^{\delta,\overline{\vt}}(\varrho,\vartheta)  (\tau, \cdot) }\Big]^{n-1}\int_{\mt}\sigma_{\varepsilon,\delta}\dx\\
&+ n\int_{\tau_1}^{\tau_2}\Big[  \intQ{ E_H^{\delta,\overline{\vt}}(\varrho,\vartheta)  (\tau, \cdot) }\Big]^{n-1}\bigg(\int_{Q}\ep \vartheta^5\dx+\int_{\partial Q}\frac{d(\vt)}{\vt}\Big(\vt-\overline{\vt}\Big)(\vt-\Theta_0)\dH\bigg)  \dt\\
&+ 2\ep n\overline\vt\E\int_{\tau_1}^{\tau_2}\Big[  \intQ{ E_H^{\delta,\overline{\vt}}(\varrho,\vartheta)  (\tau, \cdot) }\Big]^{n-1}\intO{ \left[  \delta\vr^2 + \frac{\delta \Gamma }{\Gamma - 1} \vr^{\Gamma} +\frac{1}{2}\vr|\bfu|^2\right] }  \dt\\
&+2\varepsilon n \int_{\tau_1}^{\tau_2} \Big[  \intQ{ E_H^{\delta,\overline{\vt}}(\varrho,\vartheta)  (\tau, \cdot) }\Big]^{n-1}\int_{Q}\varrho\Big( \frac{p_M(\varrho, \vartheta)}{\varrho\vt}
+\frac{ e_{M,\delta}(\varrho, \vartheta)}{\vt} -s_{M,\delta}(\varrho, \vartheta) 
 \Big)\dxt\\
&= \,\Big[\intQ{  E_H^{\delta,\overline{\vt}}(\varrho,\vartheta)(\tau_1,\cdot) }\Big]^n\\&+n\varepsilon\int_{\tau_1}^{\tau_2}\Big[  \intQ{ E_H^{\delta,\overline{\vt}}(\varrho,\vartheta)  (\tau, \cdot) }\Big]^{n-1}\int_{Q}\frac{\overline{\vt}}{\vartheta^2}\bigg(e_{M,\delta}(\varrho,\vartheta)+\varrho\frac{\partial e_M}{\partial\varrho}(\varrho,\vartheta)\bigg)\nabla\varrho\cdot\nabla\vartheta\dx\dt\\
&+n\int_{\tau_1}^{\tau_2}\Big[  \intQ{ E_H^{\delta,\overline{\vt}}(\varrho,\vartheta)  (\tau, \cdot) }\Big]^{n-1}\int_{Q}\bigg(\frac{\delta}{\vartheta^2}+\varepsilon\overline{\vt} \vartheta^4\bigg)\dx\dt\\
&+n\int_{\tau_1}^{\tau_2}\Big[  \intQ{ E_H^{\delta,\overline{\vt}}(\varrho,\vartheta)  (\tau, \cdot) }\Big]^{n-1} \int_{Q}\varrho{\mathbb F}_{\ep} (\varrho,\vartheta,\bfu)\cdot\bfu\,\dd W\dx\\&+ \frac{n}{2}\int_{\tau_1}^{\tau_2}
\Big[  \intQ{ E_H^{\delta,\overline{\vt}}(\varrho,\vartheta)  (\tau, \cdot) }\Big]^{n-1}\bigg(
\intQ{ \sum_{k \geq 1} \varrho | {\bf F}_{k,\ep} (\varrho,\vartheta,{\bf u}) |^2  } \bigg) {\rm d}t\\
&+ \ep M_\ep n\int_{\tau_1}^{\tau_2} \Big[  \intQ{ E_H^{\delta,\overline{\vt}}(\varrho,\vartheta)  (\tau, \cdot) }\Big]^{n-1} \intO{ \left( 2\delta \vr + \frac{\delta \Gamma}{\Gamma - 1} \vr^{\Gamma - 1}+\frac{1}{2}|\bfu|^2\right) }  \dt\\
&+ \ep M_\ep n\overline\vt\int_{\tau_1}^{\tau_2} \Big[  \intQ{ E_H^{\delta,\overline{\vt}}(\varrho,\vartheta)  (\tau, \cdot) }\Big]^{n-1} \intO{ \left( \frac{p_M(\varrho, \vartheta)}{\varrho\vt}+\frac{e_{M,\delta}(\varrho, \vartheta)}{\vt}
-
s_{M,\delta}(\varrho, \vartheta) \right) }  \dt\\
&+\frac{n(n-1)}{2}\int_{\tau_1}^{\tau_2} \Big[  \intQ{ E_H^{\delta,\overline{\vt}}(\varrho,\vartheta)  (\tau, \cdot) }\Big]^{n-2}\sum_{k=1}^\infty\bigg( \int_{Q}\varrho{\vc F_{k,\ep}} (\varrho,\vartheta,\bfu)\cdot\bfu\dx\bigg)^2\dt\\
&=:(I)+(II)+\dots+ (VIII).
\end{aligned}
\end{align}
Let us first consider the terms on the left-hand side. We have
\begin{align*}
n\int_{\tau_1}^{\tau_2}&\Big[  \intQ{ E_H^{\delta,\overline{\vt}}(\varrho,\vartheta)  (\tau, \cdot) }\Big]^{n-1}\int_{\partial Q}\frac{d(\vt)}{\vt}\Big(\vt-\overline{\vt}\Big)(\vt-\Theta_0)\dH\dt\\&\geq
n\int_{\tau_1}^{\tau_2}\Big[  \intQ{ E_H^{\delta,\overline{\vt}}(\varrho,\vartheta)  (\tau, \cdot) }\Big]^{n-1}\int_{\partial Q}\vt^2\dH\dt-cn\int_{\tau_1}^{\tau_2}\Big[  \intQ{ E_H^{\delta,\overline{\vt}}(\varrho,\vartheta)  (\tau, \cdot) }\Big]^{n-1}\\
&\geq n\int_{\tau_1}^{\tau_2}\Big[  \intQ{ E_H^{\delta,\overline{\vt}}(\varrho,\vartheta)  (\tau, \cdot) }\Big]^{n-1}\int_{\partial Q}\vt^2\dH\dt\\&-\kappa n\int_{\tau_1}^{\tau_2}\Big[  \intQ{ E_H^{\delta,\overline{\vt}}(\varrho,\vartheta)  (\tau, \cdot) }\Big]^{n}\dt-c_\kappa(\tau_2-\tau_1)
\end{align*}
for all $\kappa>0$ as well as
\begin{align*}
\int_{\tau_1}^{\tau_2}&\Big[  \intQ{ E_H^{\delta,\overline{\vt}}(\varrho,\vartheta)  (\tau, \cdot) }\Big]^{n-1}\intO{ \left[  \delta\vr^2 + \frac{\delta \Gamma }{\Gamma - 1} \vr^{\Gamma} +\frac{1}{2}\vr|\bfu|^2\right] }  \dt\\
&\geq \underline c\int_{\tau_1}^{\tau_2}\Big[  \intQ{ E_H^{\delta,\overline{\vt}}(\varrho,\vartheta)  (\tau, \cdot) }\Big]^{n}\dt-c(\tau_2-\tau_1)\\&-c\int_{\tau_1}^{\tau_2}\Big[  \intQ{ E_H^{\delta,\overline{\vt}}(\varrho,\vartheta)  (\tau, \cdot) }\Big]^{n-1}\int_Q \big(\vt^4+\overline\vt\vr\log(\vt)\big) \dt\\
&\geq \underline c\int_{\tau_1}^{\tau_2}\Big[  \intQ{ E_H^{\delta,\overline{\vt}}(\varrho,\vartheta)  (\tau, \cdot) }\Big]^{n}\dt-c_\kappa(\tau_2-\tau_1)\\&-\kappa\int_{\tau_1}^{\tau_2}\Big[  \intQ{ E_H^{\delta,\overline{\vt}}(\varrho,\vartheta)  (\tau, \cdot) }\Big]^{n-1}\int_Q \Big(\ep\vt^5+\delta\vr^2+\delta\frac{1}{\vt^3}\Big) \dt
\end{align*}
due to \eqref{est:H}.
Finally, due to \eqref{3.32}--\eqref{3.39},
\begin{align*}
\frac{p_M(\varrho, \vartheta)}{\vt}
+\frac{\vr e_{M,\delta}(\varrho, \vartheta)}{\vt} -\vr s_{\delta,M}(\varrho, \vartheta) 
\end{align*}
is bounded from below by a negative constant, such that the corresponding term can be bounded from below by $-c(\tau_2-\tau_1)$.\\
Using \eqref{P-1}, $\int_Q\varrho\dx=M_\ep\leq M_0$ and \eqref{est:H} the terms $(V)$ and $(VIII)$ can be bounded by
\begin{align*}
\int_{\tau_1}^{\tau_2} &\Big[  \intQ{ E_H^{\delta,\overline{\vt}}(\varrho,\vartheta)  (\tau, \cdot) }\Big]^{n-1}+\int_{\tau_1}^{\tau_2}\Big[  \intQ{ E_H^{\delta,\overline{\vt}}(\varrho,\vartheta)  (\tau, \cdot) }\Big]^{n-2}\int_Q \vr|\bfu|^2\dx \dt\\
&\leq\,c\int_{\tau_1}^{\tau_2} \Big[  \intQ{ E_H^{\delta,\overline{\vt}}(\varrho,\vartheta)  (\tau, \cdot) }\Big]^{n-1}+c(\tau_2-\tau_1)\\&+c\int_{\tau_1}^{\tau_2}\Big[  \intQ{ E_H^{\delta,\overline{\vt}}(\varrho,\vartheta)  (\tau, \cdot) }\Big]^{n-1}\int_Q \overline\vt\vr\log(\vt) \dt\\
&\leq\kappa \int_{\tau_1}^{\tau_2} \Big[  \intQ{ E_H^{\delta,\overline{\vt}}(\varrho,\vartheta)  (\tau, \cdot) }\Big]^{n}+c_\kappa(\tau_2-\tau_1)\\
&+\kappa\int_{\tau_1}^{\tau_2}\Big[  \intQ{ E_H^{\delta,\overline{\vt}}(\varrho,\vartheta)  (\tau, \cdot) }\Big]^{n-1}\int_Q \Big(\ep\vt^5+\delta\vr^2+\delta\frac{1}{\vt^3}\Big) \dt,
\end{align*}
where $\kappa>0$ is arbitrary. Clearly, $(IV)$ vanishes after taking expectations. On account of
\eqref{3.32}--\eqref{3.39} we have
\begin{align}\label{eq:pesM}
\frac{p_M(\varrho, \vartheta)}{\varrho\vt}+\frac{e_{M,\delta}(\varrho, \vartheta)}{\vt}
-
s_{M,\delta}(\varrho, \vartheta) \lesssim 1+\frac{\vr^{2/3}}{\vt}\leq \kappa \vr^\Gamma+ \kappa\frac{1}{\vt^3}+c_\kappa.
\end{align}
Consequently, the estimate for $(VII)$ is analogous to one for $(V)$ and $(VIII)$ above.
We quote from \cite[equ. (3.107)]{F}
\begin{align}\label{m140}
\begin{aligned}(II)\leq\ep  \int_{Q}& \frac{1}{\vartheta^2}
\Big| e_M(\varrho, \vartheta) + \varrho \frac{\partial e_M (\varrho,
\vartheta)}{\partial \varrho } \Big| |\Grad \varrho| |\Grad
\vartheta| \dx \\
&\leq \frac{1}{2}  \intQ{ \Big[ \delta \Big( \vartheta^{\Gamma - 2} +
\frac{1}{\vartheta^3} \Big) |\Grad \vartheta |^2 + \frac{\ep
\delta}{\vartheta} \Big( \Gamma \varrho^{\Gamma - 2} + 2 \Big)
|\Grad \varrho |^2 \Big] }
\end{aligned}
\end{align}
provided we choose $\ep = \ep(\delta) > 0$ small enough. Finally, we clearly
have
\begin{align*}
(III)&\leq \kappa\int_{\tau_1}^{\tau_2}\Big[  \intQ{ E_H^{\delta,\overline{\vt}}(\varrho,\vartheta)  (\tau, \cdot) }\Big]^{n-1}\int_Q \Big(\ep\vt^5+\delta\frac{1}{\vt^3}\Big) \dt\\&+c_\kappa\int_{\tau_1}^{\tau_2}\Big[  \intQ{ E_H^{\delta,\overline{\vt}}(\varrho,\vartheta)  (\tau, \cdot) }\Big]^{n-1}\dt\\
&\leq \kappa\int_{\tau_1}^{\tau_2}\Big[  \intQ{ E_H^{\delta,\overline{\vt}}(\varrho,\vartheta)  (\tau, \cdot) }\Big]^{n-1}\int_Q \Big(\ep\vt^5+\delta\frac{1}{\vt^3}\Big)\dx \dt\\&+\kappa\int_{\tau_1}^{\tau_2}\Big[  \intQ{ E_H^{\delta,\overline{\vt}}(\varrho,\vartheta)  (\tau, \cdot) }\Big]^{n}\dt+c_\kappa(\tau_2-\tau_1).
\end{align*}
Combing everything and choosing $\kappa$ small enough and noticing that $\delta\frac{1}{\vt^3}\leq\sigma_{\ep,\delta}$ yields
\begin{align}\label{b14}
\begin{aligned}
&
 \E\Big[  \intQ{ E_H^{\delta,\overline{\vt}}(\varrho,\vartheta)  (\tau_2, \cdot) }\Big]^n
+D \E\int_{\tau_1}^{\tau_2}\Big[  \intQ{ E_H^{\delta,\overline{\vt}}(\varrho,\vartheta)  (\tau, \cdot) }\Big]^{n-1}\int_{\mt}\sigma_{\varepsilon,\delta}\dx\\
&+ D\E\int_{\tau_1}^{\tau_2}\Big[  \intQ{ E_H^{\delta,\overline{\vt}}(\varrho,\vartheta)  (\tau, \cdot) }\Big]^{n-1}\bigg(\int_{Q} \ep \vartheta^5\dx+\int_{\partial Q}\vt^2\dH\bigg) \dt\\
&\leq\, \E\Big[  \intQ{ E_H^{\delta,\overline{\vt}}(\varrho,\vartheta)  (\tau_1, \cdot) }\Big]^n+c(\tau_2-\tau_1)
\end{aligned}
\end{align}
for all $0\leq\tau_1<\tau_2$ with some $D>0$. We obtain in particular
\begin{align*}
&
 \E\Big[  \intQ{ E_H^{\delta,\overline{\vt}}(\varrho,\vartheta)  (\tau_2, \cdot) }\Big]^n
+D \int_{\tau_1}^{\tau_2}\E\Big[  \intQ{ E_H^{\delta,\overline{\vt}}(\varrho,\vartheta)  (\tau, \cdot) }\Big]^{n}\dx\\
&\leq\, \E\Big[  \intQ{ E_H^{\delta,\overline{\vt}}(\varrho,\vartheta)  (\tau_1, \cdot) }\Big]^n+C(\tau_2-\tau_1).
\end{align*}
Applying the Gronwall lemma from \cite[Lemma 3.1]{BFH2020} and recalling hypothesis \eqref{hypo1a}
we obtain
\begin{align*}
\begin{aligned}
\E\Big[  \intQ{ E_H^{\delta,\overline{\vt}}(\varrho,\vartheta)  (\tau, \cdot) }\Big]^n&\leq \exp(-Dt)\bigg(\E\Big[  \intQ{ E_H^{\delta,\overline{\vt}}(\varrho,\vartheta)  (0, \cdot) }\Big]^n-\frac{C}{D}\bigg)+\frac{C}{D}\\&\leq E_\infty
\end{aligned}
\end{align*}
uniformly in $\tau>0$. Using this in \eqref{b14} shows
\begin{align*}
\E\int_{0}^{\tau}\Big[  \intQ{ E_H^{\delta,\overline{\vt}}(\varrho,\vartheta)  (\tau, \cdot) }\Big]^{n-1}\bigg(\int_{\mt}\sigma_{\varepsilon,\delta}\dx+\int_{Q} \ep \vartheta^5\dx+\int_{\partial Q}\vt^2\dH\bigg) \dt
&\leq\,E_\infty(1+\tau)
\end{align*}
by possibly enlarging $E_\infty$.
\end{proof}

\subsection{Stationary solutions}
Based on Proposition \ref{prop:en} the method from \cite{IN} becomes available and we can construct a stationary solution to \eqref{L1}--\eqref{L4} following the ideas from \cite{BFH2020} to which we refer for further details.
Different to Section \ref{subsec:solution} we consider stationary solutions sitting on the space of trajectories 
that are defined on the real line $R$ rather than the interval $[0,\infty)$.
We will call them \emph{entire stationary solutions}. This construction is clearly stronger
and hence we obtain also stationary solutions in the sense of Definitions \ref{D2}.

Clearly, Definition \ref{def:sol} can be easily modified for solutions $\big((\Omega,\mf,(\mf_t)_{t\geq-T},\prst),\varrho,\vartheta,\bfu,W)$
being defined
on $[-T,\infty)$ for some $T>0$. An \emph{entire solution} is than an object
$$\big((\Omega,\mf,(\mf_t)_{t\in R},\prst),\varrho,\vartheta,\bfu,W)$$
which is a solution on $[-T,\infty)$ for any $T>0$. It takes values in the trajectory space
\begin{align*}
\mathcal{T} &= \mathcal{T}_\vr\times \mathcal{T}_\vt\times \mathcal{T}_{\bfu}\times \mathcal{T}_W,\\
\mathcal{T}_\vr&=\left( L^2_{\rm loc}(R; W^{1,2} (Q; R^d) ), w \right)\cap C_{\rm weak, loc}(R; L^\Gamma(Q)),\\ \mathcal{T}_{\vt}&=\left( L^2_{\rm loc}(R; W^{1,2} (Q; R^d) ), w \right)\cap \left( L^\infty_{\rm loc}(R; L^{4} (Q) ), w^* \right)\\
\mathcal{T}_{\bfu}&=\left( L^2_{\rm loc}(R; W^{1,2}_0 (Q; R^d) ), w \right) ,\quad
\mathcal{T}_W= C_{\rm loc,0}(R; \mathfrak{U}_0 ),
\end{align*}
where $C_{{\rm loc}, 0}$ denotes the space of continuous functions vanishing at $0$. We denote by $\mathfrak{P}(\mathcal{T})$ the set of Borel probability measures on $\mathcal{T}$.

We say that an entire solution to \eqref{L1}-\eqref{L4} of the problem \eqref{L1}--\eqref{L4} is \emph{stationary} if its law 
$\mathcal{L}_{\mathcal{T}}[\vr,\vt, \vu,  W] $
is shift invariant in the trajectory space $\mathcal{T} $,
meaning
\[
\mathcal{L}_{\mathcal{T}} \left[ \mathcal{S}_\tau [\vr,\vt, \vu,  W] \right] = \mathcal{L}_{\mathcal{T}}[\vr,\vt, \vu,  W]
\ \mbox{for any}\ \tau \in R,
\]
with the time shift operator
\[
\mathcal{S}_\tau [\vr,\vt, \vu,  W] (t) = 
[\vr (t + \tau),\vt(t+\tau), \vu(t + \tau),  W (t + \tau)-W(\tau)], \ t \in R, \ \tau \in R.
\]

\begin{Proposition} \label{prop:main}
Let the assumptions of Theorem \ref{Tm1} be valid and let $\ep\leq\ep_0$ and $\delta\overline\vt>0$ be given.
 Let 
$$\big((\Omega,\mf,(\mf_t)_{t\geq0},\prst),\varrho,\vt,\bfu,W)$$ be a dissipative martingale solution of the problem 
\eqref{L1}--\eqref{L4} (in the sense of Definition \ref{def:sol} with the obvious modifications) such that 
\begin{equation} \label{hypo1}
{\rm ess} \limsup_{t \to 0 +} \expe{ E_H^{\delta,\overline\vt}(t)^4 } < \infty.
\end{equation}
Then there is a sequence $T_n \to \infty$ and an entire stationary solution 
$$\big((\tilde{\Omega},\tilde{\mf},(\tilde{\mf}_t)_{t\in R},\tilde{\prst}),\tvr,\vt,\tilde{\vu}, 
\tilde{W})$$ 
such that 
\[
\frac{1}{T_n} \int_0^{T_n} \mathcal{L}_{\mathcal{T}}\left[ \mathcal{S}_t \left[ \vr,\vt, {\vu},  {W} \right] \right] \dt \to 
\mathcal{L}_{\mathcal{T}} \left[ \tvr, \tilde{\vu},\tilde{\vt}, \tilde{W} \right] \ \mbox{narrowly as}\ 
n \to \infty.
\]
\end{Proposition}
\begin{proof}
Let $[\varrho, \tmmathbf{u}, W]$ be a dissipative martingale solution on $[0,\infty)$ defined on some stochastic basis $(\Omega,\mathfrak{F},(\mathfrak{F}_{t})_{t\geq0},\prst)$ and satisfying \eqref{hypo1}.
We define the probability measures
\begin{equation}\label{eq:6}
\nu_S \equiv \frac{1}{S} \int_0^S \mathcal{L}_{\mathcal{T}} (S_t [\varrho,\vt, \tmmathbf{u},
   W]) \dt \in \mathfrak{P} (\mathcal{T}) .
   \end{equation}
We tacitly regard functions defined on  time intervals $[-t,\infty)$ as trajectories on $R$ by extending them to $s\leq-t$ by the value at $-t$.
As in \cite[Prop. 5.1]{BFH2020} we can show that
the family of measures $\{\nu_{S};\,S>0\}$
is tight on $\mathcal{T}$. In fact, Proposition \ref{prop:en} yields

$$
\expe{ \sup_{s \in [-T,T]} \mathcal{E}_\delta^m (s+t) } + \expe{ \int_{-T\vee -t}^T  \intQ{ \big(|\nabla\vr|^2+|\nabla\vt|^2+|\Grad \vu|^{2}\big)(s+t)}   
\,\D s }\lesssim \expe{ \mathcal{E}^{m}(0)} +c.
$$
This gives the same bounds on $\vr$ and $\bfu$ as in \cite{BFH2020} and we control additionally
$$
\expe{ \sup_{s \in [-T,T]} \bigg(\int_Q\vt^4\dx\bigg)^m (s+t) } + \expe{ \int_{-T\vee -t}^T  \intQ{ |\nabla\vt|^2(s+t)}}   
$$
which implies tightness of $\frac{1}{S} \int_0^S \mathcal{L}_{\mathcal{T}} (S_t [\vt]) \dt$.
Note also that we have control of $\nabla\vr$ due to $\ep>0$ which is different from \cite{BFH2020}.

Due to \cite[Lemma 5.2]{BFH2020}, if the narrow limit of
\[  \nu_{\tau, S_{n}} \equiv \frac{1}{S_{n}} \int_0^{S_{n}} \mathcal{L} (S_{t + \tau}
   [\varrho, \tmmathbf{u}, W]) \dt  \]
in
$\mathfrak{P} (\mathcal{T})$ as $n \rightarrow \infty$ exists for some $\tau=\tau_0 \in R$ then it exists for all $\tau \in R$ and is independent of the choice
of $\tau$. Applying Jakubowski--Skorokhod's theorem \cite{jakubow}, we infer the existence of a sequence $S_{n}\to\infty$ and $\nu\in\mathfrak{P} (\mathcal{T})$ so that $\nu_{0,
S_n} \rightarrow \nu$ narrowly in $\mathfrak{P} (\mathcal{T})$ as well as $\nu_{\tau,
S_n} \rightarrow \nu$ narrowly for all $\tau \in R$. Accordingly, the limit
measure $\nu$ is shift invariant in the sense that for every $G\in BC(\mathcal{T})$ and every $\tau\in R$ we have
$\nu(G\circ S_{\tau})=\nu(G).$
To conclude the proof of Theorem~\ref{prop:main}, it remains to show that $\nu$ is a law of an entire solution to \eqref{L1}--\eqref{L4}.

First of all, we can argue as in \cite[Prop. 5.3]{BFH2020} to show that for any $S>0$
  $$
\nu_{S}\equiv \frac{1}{S}\int_{0}^{S}\mathcal{L}(S_{t}[\varrho, \tmmathbf{u}, W])\dt\in\mathfrak{P}(\mathcal{T})
  $$ 
   is a  dissipative martingale solution on $(- T, \infty)$, provided $(\varrho,\vartheta, \tmmathbf{u}, W)$ is a dissipative martingale solution on $(-T,\infty)$ defined on some probability space $(\Omega,\mathcal{F},\prst)$. The idea is to use that \eqref{L1}, \eqref{L2} and \eqref{L4}
   can be written as measurable mappings on the paths space (see the proof of \cite[Prop. 5.3]{BFH2020} for how to include the stochastic integral).
  Unfortunately, this is not true for the quantities hidden in $\sigma_{\ep,\delta}$ appearing in\eqref{L4} and \eqref{b14}. However,
   they are measurable on a subset, where the laws $\mathcal{L}(S_{t}[\varrho, \tmmathbf{u}, W])$ are supported. Recall from \eqref{eq:2711a} that
$\sigma_{\ep,\delta}$ belongs a.s. to $L^1$ in space and locally in time for any solution. This is enough to arrive at the same conclusion.

To finish the proof we  argue that the limit $\nu$ is the law of  
an entire solution to \eqref{L1}--\eqref{L4}. 
Now we consider the measures $\nu_{\tau, S_n - \tau}$, $n=1,2, \dots$, and $\tau>0$. According to the previous considerations, $\nu_{\tau,S_{n}-\tau}$ is a dissipative martingale solution to \eqref{L1}--\eqref{L4} on $[-\tau,\infty)$ and the narrow limit as $n\to\infty$ exists and equals to $\nu$.
Now we take a sequence $\tau_m \rightarrow \infty$ and choose a diagonal sequence such that
\[ \nu_{\tau_m, S_{n (m)} - \tau_m} \rightarrow \nu \quad \tmop{as} \quad m \rightarrow
   \infty . \]

   Applying  Jakubowski--Skorokhod's theorem, we obtain a sequence of approximate processes $[\tilde{\varrho}_m, \tilde{\tmmathbf{u}}_m, \tilde{W}_m]$ converging a.s. to a process $[\tilde{\varrho}, \tilde{\tmmathbf{u}}, \tilde{W}] $ in the topology of $\mathcal{T}$. Moreover, the law of $[\tilde{\varrho}_m, \tilde{\tmmathbf{u}}_m, \tilde{W}_m]$ is $\nu_{\tau_{m},S_{n(m)}-\tau_{m}}$ and necessarily the law of $[\tilde{\varrho}, \tilde{\tmmathbf{u}}, \tilde{W}] $ is $\nu$. 
   By \cite[Thm. 2.9.1]{BFHbook} it follows that equations \eqref{L1}--\eqref{L4} as well as 
   \eqref{b2a} also hold on the new probability space.
   The limit procedure on this level is quite easy due to the artificial viscosity: By
   definition of $\mathcal T_\vr$ the sequence $\tilde\vr_n$ is compact on $L^\Gamma$. Moreover, the strong convergence of $\tilde\vt_n$ can be proved exactly as in the deterministic existence theory
   (see \cite[Sec. 3.5.3]{F}). This is enough to pass to the limit in all nonlinearities in
   \eqref{L1}, \eqref{L2} and \eqref{L4}. The terms in \eqref{L3} and \eqref{b2a}
   which are not compact (those related to the quantity $\sigma_{\ep,\delta}$) are convex functions
   and hence can be dealt with by lower-semicontinuity.
\end{proof}

\section{Asymptotic limit}
\label{sec:lim}
In this section we pass to the limit and the artificial viscosity and the artificial pressure respectively. They crucial point is a uniform-in-time estimate, see \eqref{d5} and \eqref{d5'} below, which preserves stationarity in the limit. It has to be combined with pressure estimates which differ on both levels. The key ingredient for estimates \eqref{d5} and \eqref{d5'} is the non-homogeneous boundary condition for the temperature, cf. \eqref{beq:1}.

\subsection{The vanishing viscosity limit}
In this section we start with a stationary solution $(\vr,\vt,\bfu)$ to \eqref{L1}--\eqref{L4} existence of which 
is guaranteed by Proposition \ref{prop:main}. We prove uniform-in-time estimates and pass subsequently to the limits in $\ep$ and $\delta$.\\
The entropy balance \eqref{L3} yields after taking expectations and using stationarity
\begin{align*} 
\E\int_{Q}&
\frac{1}{\vartheta}\Big[\mathbb S(\vartheta,\nabla\bfu):\nabla\bfu+\frac{\kappa_\delta(\vartheta)}{\vartheta}|\nabla\vartheta|^2+\delta\frac{1}{\vt^2}\Big]
\,\dif x \\
&\leq\E\int_Q\big(-2\varepsilon \varrho +2\ep M_\ep\big)\frac{1}{\vt}\Big(
\frac{p_M(\varrho, \vartheta)}{\varrho}+ e_{M,\delta}(\varrho, \vartheta) - \vartheta
s_{M,\delta}(\varrho, \vartheta)  \Big)\,\varphi\dx\\
&+\E\int_Q\ep \vartheta^4 \dx+\E\int_{\partial Q}\frac{d(\vt)}{\vt}(\vt-\Theta_0)\dH.
\end{align*}
On account of \eqref{eq:pesM} the first two terms can be bounded by
\begin{align*}
c\,\E\Big[  \intQ{ \delta\vr^\Gamma }+1\Big]+\frac{1}{4}\E\int_{Q}\delta\frac{1}{\vt^3}\dx.
\end{align*}
 The estimate is independent of $\delta$ if we choose $\ep\leq \delta$.
Similarly, we obtain
\begin{align*}
\E\int_{\partial Q}\frac{d(\vt)}{\vt}(\vt-\Theta_0)\dH&\leq\,c\, \E\int_{\partial Q}\big(\vt+|\nabla\vt|\big)+c\\
&\leq c\,\E\Big[  \intQ{\vt^4 }+1\Big]+\frac{1}{4}\E\int_{Q}\frac{\kappa_\delta(\vartheta)}{\vartheta}|\nabla\vartheta|^2
\,\dif x
\end{align*}
using \eqref{beq:2} and the trace theorem. In conclusion,
\begin{align} \label{stat:entropy}
\begin{aligned}
\E\int_{Q}
\frac{1}{\vartheta}\Big[\mathbb S(\vartheta,\nabla\bfu):\nabla\bfu+\frac{\kappa_\delta(\vartheta)}{\vartheta}|\nabla\vartheta|^2+\delta\frac{1}{\vt^2}\Big]
\,\dif x&\lesssim \E\Big[  \intQ{ \big(\delta\vr^\Gamma+\vt^4\big)}\Big]+1\\
&\lesssim \E\Big[  \intQ{ E_H^{\delta,\overline{\vt}}(\varrho,\vartheta)}\Big]+1
\end{aligned}
\end{align}
independently of $\ep$ and $\delta$ recalling also \eqref{est:H} and $\int_Q\vr\dx\leq M_\ep\leq M_0$. By \eqref{est:H} this implies for any $\xi>0$
\begin{align*}
\E\Big[  \intQ{ E_H^{\delta,\overline{\vt}}(\varrho,\vartheta)}\Big]&\leq\,c\, \E\Big[  \intQ{ \big(\delta\vr^\Gamma+\vr^{5/3}+\vt^4+\delta\vr|\log(\vt)|\big)}+1\Big]\\
&\leq\,c \E\Big[  \intQ{ \big(\delta\vr^\Gamma+\vr^{5/3}+\vt^4\big)}\Big]+\xi\, \E\Big[  \intQ{ \delta\frac{1}{\vt^3}}\Big]+c_\xi\\
&\leq\,c \E\Big[  \intQ{ \big(\delta\vr^\Gamma+\vr^{5/3}+\vt^4\big)}\Big]+c\xi\,\E\Big[  \intQ{ E_H^{\delta,\overline{\vt}}(\varrho,\vartheta)}\Big]+c_\xi
\end{align*}
such that
\begin{align}\label{eq:2811}
\E\Big[  \intQ{ E_H^{\delta,\overline{\vt}}(\varrho,\vartheta)}\Big]
&\lesssim \E\Big[  \intQ{ \big(\delta\vr^\Gamma+\vr^{5/3}+\vt^4\big)}\Big]+1
\end{align}
independently of $\ep$ and $\delta$.

In \eqref{b2a} we choose $n=2$, apply expectations and use stationarity to obtain
 \begin{align} \label{b2A}
\begin{aligned}
&2\overline{\vt} \E\Big[  \intQ{ E_H^{\delta,\overline{\vt}}(\varrho,\vartheta)  (\tau, \cdot) }\Big]\int_{\mt}\sigma_{\varepsilon,\delta}\dx\\
&+ 2\E\Big[  \intQ{ E_H^{\delta,\overline{\vt}}(\varrho,\vartheta)  (\tau, \cdot) }\Big]\int_{Q} \ep \vartheta^5\dx \\
&+2\E\Big[  \intQ{ E_H^{\delta,\overline{\vt}}(\varrho,\vartheta)  (\tau, \cdot) }\Big]\int_{\partial Q}\frac{d}{\vt}\Big(\vt-\overline{\vt}\Big)(\vt-\Theta_0)\dH\\
&+ 4\ep \overline\vt\E\Big[  \intQ{ E_H^{\delta,\overline{\vt}}(\varrho,\vartheta)  (\tau, \cdot) }\Big]\intO{ \left[  \delta\vr^2 + \frac{\delta \Gamma }{\Gamma - 1} \vr^{\Gamma} +\frac{1}{2}\vr|\bfu|^2\right] }  \\
&+4\varepsilon  \E\Big[  \intQ{ E_H^{\delta,\overline{\vt}}(\varrho,\vartheta)  (\tau, \cdot) }\Big]\int_{Q}\varrho\Big( \frac{p(\varrho, \vartheta)}{\varrho\vt}
+\frac{ e_\delta(\varrho, \vartheta)}{\vt} -s_\delta(\varrho, \vartheta) 
 \Big)\dx\\
&\leq 2\varepsilon\E\Big[  \intQ{ E_H^{\delta,\overline{\vt}}(\varrho,\vartheta)  (\tau, \cdot) }\Big]\int_{Q}\frac{\overline{\vt}}{\vartheta^2}\bigg(e_{M,\delta}(\varrho,\vartheta)+\varrho\frac{\partial e_M}{\partial\varrho}(\varrho,\vartheta)\bigg)\nabla\varrho\cdot\nabla\vartheta\dx\\
&+
\E\Big[  \intQ{ E_H^{\delta,\overline{\vt}}(\varrho,\vartheta)  (\tau, \cdot) }\Big]\bigg(
\intQ{ \sum_{k \geq 1} \varrho | {\bf F}_{k,\ep} (\varrho,\vartheta,{\bf u}) |^2  } \bigg)\\
&+ 2\ep M_\ep\E\Big[  \intQ{ E_H^{\delta,\overline{\vt}}(\varrho,\vartheta)  (\tau, \cdot) }\Big] \intO{ \left( 2\delta \vr + \frac{\delta \Gamma}{\Gamma - 1} \vr^{\Gamma - 1}+\frac{1}{2}|\bfu|^2\right) } \\
&+ 2\ep M_\ep\overline\vt\E\Big[  \intQ{ E_H^{\delta,\overline{\vt}}(\varrho,\vartheta)  (\tau, \cdot) }\Big] \intO{ \left( \frac{p(\varrho, \vartheta)}{\varrho\vt}+\frac{e_\delta(\varrho, \vartheta)}{\vt}
-
s_\delta(\varrho, \vartheta) \right) } \\
&+\sum_{k=1}^\infty\E\bigg( \int_{Q}\varrho\vc F_{k,\ep} (\varrho,\vartheta,\bfu)\cdot\bfu\dx\bigg)^2\\
&=:(I)+(II)+(III)+(IV)+(V).
\end{aligned}
\end{align}
Arguing as in the proof of Proposition \ref{prop:en} but paying attention to the $\ep$- and $\delta$-dependence we have
\begin{align*}
(I)&\leq\frac{1}{4}\E\Big[  \intQ{ E_H^{\delta,\overline{\vt}}(\varrho,\vartheta)  (\tau, \cdot) }\Big]\int_Q\sigma_{\ep,\delta}\dx,\\
(II)&\leq\,c\, \E\Big[  \intQ{ E_H^{\delta,\overline{\vt}}(\varrho,\vartheta)  (\tau, \cdot) }\Big],\\
(III)&\leq\ep \overline\vt\E\Big[  \intQ{ E_H^{\delta,\overline{\vt}}(\varrho,\vartheta)  (\tau, \cdot) }\Big]\intO{ \left[  \delta\vr^2 + \frac{\delta \Gamma }{\Gamma - 1} \vr^{\Gamma}\right] }\\
&+\,c\, \E\Big[  \intQ{ E_H^{\delta,\overline{\vt}}(\varrho,\vartheta)  (\tau, \cdot) }\Big]+\frac{1}{4}\E\Big[  \intQ{ E_H^{\delta,\overline{\vt}}(\varrho,\vartheta)  (\tau, \cdot) }\Big]\int_Q\sigma_{\ep,\delta}\dx,\\
(IV)&\leq\ep \overline\vt\E\Big[  \intQ{ E_H^{\delta,\overline{\vt}}(\varrho,\vartheta)  (\tau, \cdot) }\Big]\intO{ \frac{\delta \Gamma }{\Gamma - 1} \vr^{\Gamma} }+c\\
&+\frac{1}{4}\E\Big[  \intQ{ E_H^{\delta,\overline{\vt}}(\varrho,\vartheta)  (\tau, \cdot) }\Big]\int_Q\sigma_{\ep,\delta}\dx,\\
(V)&\leq\,c\, \E\Big[  \intQ{ E_H^{\delta,\overline{\vt}}(\varrho,\vartheta)  (\tau, \cdot) }\Big].
\end{align*}
Again these estimates are also uniform in $\delta$ if we choose $\ep$ small enough compared to $\delta$.
Recalling \eqref{est:H} and $\int_Q\vr\dx\leq M_\ep\leq M_0$ we thus obtain
\begin{align*}
&\| \vt \|_{L^6(Q)} \aleq \| \vt \|_{W^{1,2}(Q)}
\aleq 1+
\int_Q \frac{\Ov{\vt}}{\vt} \frac{\kappa(\vt)}{\vt} |\Grad \vt|^2 {\rm d}x + \int_{\partial Q}\vt^2  \dH,\\
&\expe{ \int_Q  E_H^{\delta,\overline{\vt}}(\vr,\vt,\bfu) \ \dx \| \vt \|_{L^6(Q)} } \ageq
\expe{ \| \vt \|_{L^4(Q)}^4 \| \vt \|_{L^6(Q)} }-\expe{ \| \vt \|_{L^6(Q)} } \\\qquad\qquad&\geq \expe{ \| \vt \|_{L^{30/7}(Q)}^5 }-\expe{  \| \vt \|_{L^6(Q)} }. 
\end{align*}
This, inserted in left-hand side of \eqref{b2A} yields
\begin{align}\label{eq:key}
\begin{aligned}
 \expe{ \| \vt \|_{L^{30/7}(Q)}^{30/7} }&+\expe{ \int_Q  E_H^{\delta,\overline{\vt}}(\vr,\vt,\bfu) \ \dx \int_Q\sigma_{\ep,\delta}\dx }\\
 &\aleq  \E\Big[  \intQ{ E_H^{\delta,\overline{\vt}}(\varrho,\vartheta) }\Big]+1.
 \end{aligned}
\end{align}
independently of $\ep$ and $\delta$ using also \eqref{stat:entropy}.

In order to close the estimate why apply pressure estimates which now can depend on $\delta$. Let us introduce the so--called Bogovskii operator $\mathcal{B}$ enjoying the following properties: 
\begin{align}\label{def:bog}
\begin{aligned}
\mathcal{B} : L^q_0 (Q) &\equiv \left\{ f \in L^q(Q) \ \Big| \ \int_Q f \dx = 0 \right\}
\to W^{1,q}_0(Q, R^d),\ 1 < q < \infty,\\
\Div \mathcal{B}[f] &= f ,\\ 
\| \mathcal{B}[f] \|_{L^r(Q)} &\aleq \| \vc{g} \|_{L^r(Q; R^d)} \ \mbox{if}\ 
f = \Div \vc{g}, \ \vc{g} \cdot \vc{n}|_{\partial Q} = 0, \ 1 < r < \infty,
\end{aligned}
\end{align}
see \cite[Chapter 3]{Gal2011} or \cite{GH}.
Arguing as in \cite[Section 5]{BFHM} (but replacing $\nabla\Delta^{-1}$ by the Bogovskii operator $\mathcal B$) we obtain
\begin{equation*}
\begin{split}
&\expe{  \intO{ \left[p(\vr,\vt)\vr   +
\frac{1}{3}\vr^{2} |\vu|^2 \right]} }\\
& =c(M_0) \,\expe{  \intO{ \Big( p(\vr,\vt)   +
\frac{1}{3}\vr |\vu|^2  \Big) } } \\
&\quad - \expe{  \intO{ \Big( \vr \vu \otimes \vu - \frac{1}{3} \vr |\vu |^2 \mathbb{I} \Big) : \Grad \mathcal B( \vr-M_\varepsilon)  } } \\
&\quad+
\expe{  \intO{ \Big(\frac{4}{3}\mu(\vt)+\eta(\vt) \Big) \Div \vu \ \vr  }\dt }+ \expe{ \intO{ \vr \vu \cdot \mathcal B  \Div (\vr \vu )  }}\\
&\quad+ 2 \ep \,\expe{ \intO{ \vre \vue \cdot \mathcal B\left[
\vr_\varepsilon -M_\varepsilon \right] } } +\ep\,\expe{ \intO{ \vre^2 \Div\vue } }\\
&=:(I)+(II)+(III)+(IV)+(V)+(VI).
\end{split}
\end{equation*}
The terms (II) and (IV)--(VI) can be estimated as in \cite{BFHM} (note that they don't contain $\vt$). 
In fact, we have by \eqref{est:H} and the continuity of $\nabla \mathcal B$
\begin{align*}
(II)&\lesssim \E\|\sqrt{\vr}\bfu\|_{L^2_x}\|\bfu\|_{L^6_x}\|\sqrt{\vr}\nabla\mathcal B(\vr-M_\ep)\|_{L^3_x}\\
&\lesssim \E\int_{\mt}E_H(\vr,\vt,\bfu)\dx\|\nabla\bfu\|^2_{L^2_x}+\E\|\sqrt{\vr}\nabla \mathcal B(\vr-M_\ep)\|_{L^3_x}^2\\
&\lesssim\,1+ \E\intO{  E_H^{\delta,\overline{\vt}}(\vr,\vt,\bfu) }
\end{align*}
provided $\Gamma$ is large enough.
Furthermore, we obtain for some $\alpha\in(0,1)$
\begin{align*}
(IV)&\lesssim \E\|\vr\bfu\|_{L^2_x}^2\lesssim \E\|\vr\|_{L^3_x}^2\|\bfu\|_{L^6_x}^2\lesssim
\E\|\vr\|_{L^3_x}^2\|\nabla\bfu\|_{L^2_x}^2\\&\lesssim \E\|\vr\|_{L^1_x}^{2\alpha}\|\vr\|^{2(1-\alpha)}_{L^\Gamma}\|\nabla\bfu\|_{L^2_x}^2\lesssim \E\|\vr\|^{2(1-\alpha)}_{L^\Gamma}\|\nabla\bfu\|_{L^2_x}^2\\
&\lesssim \E\int_{\mt}E_H^{\delta,\overline{\vt}}(\vr,\vt,\bfu)\dx\|\nabla\bfu\|^2_{L^2_x}+\E\|\nabla\bfu\|^2_{L^2_x}\\
&\lesssim\,1+ \E\intO{  E_H^{\delta,\overline{\vt}}(\vr,\vt,\bfu) } 
\end{align*}
using again \eqref{est:H}, $\int_Q\varrho\dx\leq M_\ep\leq M_0$, \eqref{eq:key} as well as \eqref{stat:entropy}.
We can estimate $(V)$ and $(VI)$ along the same lines.
Using \eqref{m105} and \eqref{m106} we have for $\Gamma$ large enough
\begin{align*}
(III)\lesssim \expe{ \intO{\big(\vt^4+\varrho^4+|\nabla\bfu|^2\big)}}\lesssim \expe{ 1+\left( \intO{  E_H^{\delta,\overline{\vt}}(\vr,\vt,\bfu) } \right)}
\end{align*}
due to \eqref{est:H} \eqref{stat:entropy}. Obviously, the same estimate holds for
(I) such that we can conclude
\begin{align}\label{pressure:ep}
\expe{  \intO{  \big(\varrho^{\Gamma+1}+\vt^4\vr+\vr^2|\bfu|^2 \big)}}\lesssim \expe{ 1+ \intO{  \big(\varrho^\Gamma+\vt^4+\vr|\bfu|^2\big) } },
\end{align}
using also  \eqref{eq:2811}.
Obviously, we have
\begin{align*}
\intO{  \big(\varrho^\Gamma+\vt^4\big) }&\leq\,\xi\intO{  \big(\varrho^{\Gamma+1}+\vt^{30/7}\big) }+c(\xi),\\
\expe{\intO{ \vr|\bfu|^2 } }&\leq \tilde\xi\expe{\intO{ \vr^2|\bfu|^2 } }+c(\tilde\xi)\expe{\intO{ |\nabla\bfu|^2 } }\\
&\leq \tilde\xi\expe{\intO{ \vr^2|\bfu|^2 } }+c(\tilde\xi)\intO{  \big(\varrho^\Gamma+\vt^4+1\big) }
\end{align*}
for any $\xi,\tilde\xi>0$ using again \eqref{stat:entropy}.
Consequently, all terms can be absorbed in the left-hand side and we end up with
\begin{equation}\label{d5}
\expe{  \intO{ \left[\vr^{\Gamma+1}+\vt^4\vr+\vt^{30/7}   +
\vr^{2} |\vu|^2 \right]} }\leq \,c
\end{equation}
as well as
 \begin{align} \label{d5'}
& \expe{ \left( \intO{  E_H^{\delta,\overline{\vt}}(\vr,\vt,\bfu) } \right)
\int_{\mt}
\sigma_{\ep,\delta} \dx}  \leq\,c
\end{align}
using \eqref{stat:entropy}.\\

Estimates \eqref{d5} and \eqref{d5'} are sufficient to pass to the limit in \eqref{L1}--\eqref{L4} arguing as in \cite[Section 5]{BF} (in fact, one has to combine ideas from \cite{BFHbook} and \cite{F}).
In the limit $\varepsilon\rightarrow0$, we obtain the following system.

\begin{itemize}
\item {\bf Equation of continuity.}
\begin{equation} \label{L1del}
\begin{split}
\int_{0}^{\infty} &\intO{ \left[ \vr \partial_t \varphi + \vr \vu \cdot \Grad \varphi \right] } \dt = 0
\end{split}
\end{equation}
for any $\varphi \in \DC((0, \infty) \times \mt)$ $\prst$-a.s.
\item {\bf Momentum equation.}
\begin{equation} \label{L2del}
\begin{split}
\int_0^\infty \partial_t &\psi\intO{  \vr \vu \cdot \bfvarphi} \dt  + \int_0^\infty\psi  \intO{ \vr \vu \otimes \vu : \Grad \bfvarphi } \dt
+ \int_0^\infty \psi \intO{ p_\delta(\vt,\vr)  \Div \bfvarphi } \dt  \\
&-   \int_0^\infty\psi  \intO{ \mathbb{S}_\delta(\vt,\Grad \vu) : \Grad \bfvarphi} \dt 
=- \int_0^\infty \psi \intO{ \vr\mathbb{F}(\vr,\vt,\bfu) \cdot \bfvarphi } \ \D W
\end{split}
\end{equation}
for any $\psi \in \DC((0, \infty))$, $\bfvarphi \in C^\infty(\mt; \R^3)$ $\prst$-a.s.
\item {\bf Entropy balance.}
\begin{equation} \label{L3del}
\begin{aligned}
-
 \int_0^\infty &\intQ{ \left[\varrho s_\delta(\varrho,\vartheta) \partial_t \psi + \varrho s_\delta (\varrho,\vartheta)\bfu \cdot \nabla\psi \right]\,\varphi}\dif t\\\geq & \int_0^\infty\int_{Q}
\frac{1}{\vartheta}\Big[\mathbb S_\delta(\vartheta,\nabla\bfu):\nabla\bfu+\frac{\kappa_\delta(\vartheta)}{\vartheta}|\nabla\vartheta|^2+\delta\frac{1}{\vt^2}\Big]
\psi\,\dif x\,\dif t \\
&+ \int_0^\infty\intQ{ \frac{\kappa_\delta(\vartheta)\nabla\vartheta}{\vartheta} \cdot \nabla\psi }\,\dif t-\int_0^\infty\psi\int_{\partial Q}\frac{d}{\vt}(\vt-\Theta_0)\dH\dt
\end{aligned}
\end{equation}
for any $\psi \in \DC((0, \infty))$, $\varphi \in C^\infty(\mt; \R^3)$ $\prst$-a.s.
\item {\bf Total energy balance.}
\begin{equation} \label{L4del}
\begin{split}
- \int_0^T& \partial_t \psi
\bigg(\int_{Q} {\mathcal E_\delta(\varrho,\vartheta, \vu) } \dx\bigg) \dt
+\int_0^\infty\psi\int_{\partial Q}d(\vt-\Theta_0)\dH\dt \\
= &\,\psi(0) \intTor{{\mathcal E_\delta(\varrho_0,\vartheta_0, \vu_0 )} }+\int_0^T\int_{Q}\frac{\delta}{\vartheta^2}\psi\dx\dt\\
&+ \frac{1}{2} \int_0^T
\psi \bigg(
\intQ{ \sum_{k \geq 1} \varrho | {\bf F}_{k} (\varrho,\vartheta, {\bf u}) |^2  } \bigg) {\rm d}t\\
&+\int_0^T  \psi\int_{Q}\varrho{\vc{F}}(\varrho,\vartheta,\bfu)\cdot\bfu\,\dd W\dx
\end{split}
\end{equation}
for any $\psi \in \DC(0, \infty)$ $\mathbb{P}$-a.s.
\end{itemize}

To summarize, we deduce the following.

\begin{Proposition}\label{prop:1608bdel}
Let $\delta>0$ be given. Then there exists a stationary weak martingale solution $[\varrho_\delta,\vt_\delta,\bfu_\delta]$ to \eqref{L1del}--\eqref{L4del}. Moreover, we have the estimates
\begin{align} \label{est:delta1}
 \expe{ \| \vt \|_{L^{30/7}(Q)}^{30/7} }
 &+\expe{ \int_Q  E_H^{\delta,\overline{\vt}}(\vr,\vt,\bfu) \ \dx \int_Q\sigma_{\delta}\dx }\\
 &\aleq  \E\Big[  \intQ{ E_H^{\delta,\overline{\vt}}(\varrho,\vartheta) }\Big]+1,\nonumber\\
 \label{est:delta2}\E\int_Q\sigma_{\delta}\dx &\lesssim\,1+ \E\intO{  E_H^{\delta,\overline{\vt}}(\vr,\vt,\bfu) }
,\end{align}
uniformly in $\delta$, where 
\begin{align*}
\sigma_{\delta}&=\frac{1}{\vartheta}\Big[\mathbb S(\vartheta,\nabla\bfu):\nabla\bfu+\frac{\kappa(\vartheta)}{\vartheta}|\nabla\vartheta|^2+\frac{\delta}{2}\Big(\vt^{\Gamma-1}+\frac{1}{\vartheta^2}\Big)|\nabla\vartheta|^2+\delta\frac{1}{\vartheta^2}\Big].
\end{align*}
\end{Proposition}
\begin{corollary}
The solution from Proposition \ref{prop:1608bdel} satisfies the equation of continuity in the renormalised sense.
\end{corollary}

\subsection{The vanishing artificial pressure limit}
Though \eqref{est:delta1} and \eqref{est:delta2} are uniform in $\delta$, the final estimates \eqref{d5} and \eqref{d5'} are not. Again we have to close the estimate by some pressure bounds.
Let $(\vr,\vt,\bfu)$ be a stationary solution to \eqref{L1del}--\eqref{L4del} as obtained in
Proposition \ref{prop:1608bdel}.
Arguing as in \cite[Section 6]{BFHM} (replacing again $\nabla\Delta^{-1}$ by the Bogovskii operator $\mathcal B$)
we have
\begin{equation} \label{Dd5}
\begin{split}
&\expe{  \intO{ \bigg[  p_\delta(\vr,\vt) \vr^{\alpha}   +
\vr_{\delta}^{1 + \alpha} |\vu|^2 \bigg]} }  \\
& \leq
c(M_0) \left( \expe{ \intO{ \left[ \frac{1}{2} \vr |\vu|^2 + p_\delta(\vr,\vt) \right] } } + 1 \right) \\
 &\qquad+
\expe{ \intO{ \left( \frac{4}{3}\mu(\vt)+\eta(\vt) \right) \Div \vu \ \vr^\alpha  } } \\
& \qquad+ \expe{ \intO{ \left( \vr \vu \otimes \vu - \frac{1}{3} \vr |\vu |^2 \mathbb{I} \right) : \Grad \mathcal B \left[ \vr^\alpha \right] } } \\
&\qquad + \expe{ \intO{ \vr \vu \cdot \mathcal B [ \Div (\vr^\alpha \vu ) + (\alpha - 1) \vr^\alpha \Div \vu ] }}\\
&=:(I)+(II)+(III)+(IV),
\end{split}
\end{equation}
where $\alpha>0$ will be chosen sufficiently small. As in the proof of \cite[Prop. 6.1]{BFHM}
we obtain
\begin{align*}
(I)+(III)+(IV)\lesssim \expe{ \int_Q  E_H^{\delta,\overline{\vt}}(\vr,\vt,\bfu) \ \dx \int_Q|\nabla\bfu|^2\dx }+\expe{ \int_Q  E_H^{\delta,\overline{\vt}}(\vr,\vt,\bfu) \ \dx }+1.
\end{align*}
Also we see that
\begin{align*}
(III)&\lesssim \expe{ \int_Q|\nabla\bfu|^2\dx }+\expe{ \int_Q\big(\varrho^\gamma+\vt^4\big)\dx }+1\\&\lesssim \expe{ \int_Q|\nabla\bfu|^2\dx }+\expe{ \int_Q  E_H^{\delta,\overline{\vt}}(\vr,\vt,\bfu) \ \dx }+1
\end{align*}
choosing $\alpha$ small enough and using \eqref{m105} and \eqref{m106}.
Combining these estimate with \eqref{est:delta1} and \eqref{est:delta2} we conclude
\begin{align}\label{pressure:delta}
\begin{aligned}
\E\bigg[  \int_Q  \big(\delta\varrho^{\Gamma+\alpha}+&\vr^{\gamma+\alpha}+\vt^4\vr^\alpha+\vr^{1+\alpha}|\bfu|^2 \big)\dx\bigg]\\&\lesssim \expe{ 1+ \intO{  \big(\delta\vr^\Gamma+\varrho^\gamma+\vt^4+\vr|\bfu|^2\big) } },
\end{aligned}
\end{align}
recalling also \eqref{eq:2811}.
As in the proof of \eqref{d5} and \eqref{d5'} we deduce
\begin{align}\label{d5delta}
&\expe{  \intO{ \left[\vr^{\gamma+1}+\vt^4\vr+\vt^{30/7}   +
\vr^{2} |\vu|^2 \right]} }\leq \,c\\
\label{d5'delta}
& \expe{ \left( \intO{  E_H^{\delta,\overline{\vt}}(\vr,\vt,\bfu) } \right)
\int_{\mt}
|\nabla\bfu|^2 \dx}  \leq\,c,\\
\label{d5''delta}
& \expe{ \left( \intO{  E_H^{\delta,\overline{\vt}}(\vr,\vt,\bfu) } \right)
\int_{\mt}
\sigma_\delta \dx}  \leq\,c,
\end{align}
using \eqref{est:delta2}.
With estimates \eqref{d5delta} and \eqref{d5'delta} at hand we can follow the lines of \cite[Section 6]{BF} to pass to the limit
$\delta\rightarrow0$ in \eqref{L1del} and \eqref{L2del}. The limit
in \eqref{L1del} and \eqref{L2del} can be performed as in \cite[Section 7]{BF} due to \eqref{d5delta} and \eqref{d5''delta}. This finishes the proof of Theorem \ref{Tm1}.\\\ 

\section*{Compliance with Ethical Standards}\label{conflicts}

\smallskip
\par\noindent 
{\bf Funding}. 
The research of E.F.~leading to these results 
has received funding from
the Czech Sciences Foundation (GA\v CR), Grant Agreement
21--02411S.
The Institute of Mathematics of the Academy of Sciences of
the Czech Republic is supported by RVO:67985840.

\smallskip
\par\noindent
{\bf Conflict of Interest}. The authors declare that they have no conflict of interest.

\smallskip
\par\noindent
{\bf Data Availability}. Data sharing is not applicable to this article as no datasets were generated or analysed during the current study.


\begin{thebibliography}{19}
\bibitem{BF} D. Breit \& E. Feireisl: \emph{Stochastic Navier--Stokes--Fourier equations.} Indiana Univ. Math. J. 69, 911--975. (2020)
\bibitem{BFHbook} D. Breit, E. Feireisl \& M. Hofmanov\'a: \emph{Stochastically forced compressible fluid flows.} De Gruyter Series in Applied and Numerical Mathematics. De Gruyter, Berlin/Munich/Boston, 344 pp.  (2018)
\bibitem{BFH2020} D. Breit, E. Feireisl \& M. Hofmanov\'{a}: \emph{On the long time behavior of compressible fluid flows excited by random forcing.}  arXiv:2012.07476
\bibitem{BFHM} D. Breit, E. Feireisl, M. Hofmanov\'{a} \& B. Maslowski: \emph{Stationary solutions to the compressible Navier--Stokes system driven by stochastic forcing.} Probab. Theory Relat. Fields 174, 981--1032. (2019)
\bibitem{BrHo} D. Breit \& M. Hofmanov\'a: \emph{Stochastic Navier--Stokes equations for compressible fluids.} Indiana Univ. Math. J. 65, 1183--1250. (2016)
\bibitem{daprato} G. Da Prato \& J. Zabczyk: Stochastic Equations in Infinite Dimensions, Encyclopedia Math. Appl., vol. 44, Cambridge University Press, Cambridge. (1992)
\bibitem{fei3} E. Feireisl: \emph{Dynamics of Compressible Flow, Oxford Lecture  Series in Mathematics and its Applications.} Oxford University Press, Oxford. (2004)
\bibitem{F} E. Feireisl, A. Novotn\'{y}: \emph{Singular limits in thermodynamics of viscous fluids.} Birkh\"auser-Verlag, Basel. (2009)
\bibitem{FMNP} E. Feireisl, P. Mucha, A. Novotn\'{y} \& M. Pokorn\'y: \emph{Time-Periodic Solutions to the Full Navier--Stokes--Fourier System.} Arch. Rational Mech. Anal. 204, 745--786. (2012)


\bibitem{francoa} F. Flandoli: Dissipativity and invariant measures for stochastic Navier--Stokes equations. NoDEA 1,
403--423. (1994)
\bibitem{franco} F. Flandoli \& D. G\c{a}tarek: \emph{Martingale and stationary solutions for stochastic Navier--Stokes equations.} Probab. Theory Related
Fields 102, 367--391.  (1995)
\bibitem{FlMa} F. Flandoli \& B. Maslowski: \emph{Ergodicity of the 2D Navier--Stokes equation under random perturbations}
Commun. Math. Phys. 171, 119--141. (1995)
\bibitem{FlaRom1} F. Flandoli \& M. Romito: \emph{Partial regularity for the stochastic Navier--Stokes equations.} Trans. Amer.
Math. Soc. 354, 2207--2241. (2002)
\bibitem{Gal2011}
G.~P. Galdi.
\newblock {\em An introduction to the mathematical theory of the
  {N}avier-{S}tokes equations. {S}teady-state problems.}
\newblock Springer Monographs in Mathematics. Springer-Verlag, New York, 2011.
\bibitem{GH} M. Gei\ss ert, H. Heck \& M. Hieber. On the equation $\Div u = g$ and Bogovskii's operator in Sobolev spaces of negative order. In Partial differential
equations and functional analysis, volume 168 of Oper. Theory Adv. Appl.,
pages 113--121. Birkh\"auser, Basel, 2006.
\bibitem{HM} M. Hairer \& J. Mattingly: \emph{Ergodicity of the 2D Navier--Stokes Equations with Degenerate Stochastic Forcing.}
Ann. Math. 164, 993--1032. (2006)
  \bibitem{IN} K. It\^{o} a \& M. Nisio. On stationary solutions of a stochastic differential equation. J. Math. Kyoto Univ., 4:1--75.
(1964)
\bibitem{jakubow} A. Jakubowski: \emph{The almost sure Skorokhod representation for subsequences in nonmetric spaces}, Teor. Veroyatnost. i Primenen 42 (1997), no. 1, 209-216; translation in Theory Probab. Appl. 42 (1997), no. 1, 167--174. (1998)
\bibitem{Ku} S. Kuksin \& A. Shirikyan: \emph{Mathematics of Two-Dimensional Turbulence}. Cambridge Tracks of Mathematics (Book 194). Cambridge University Press, Cambridge. (2012)
%

\bibitem{SmTr} S. Smith \& K. Trivisa.
\newblock {The stochastic Navier--Stokes equations for heat-conducting, compressible fluids: global existence of weak solutions.}
\newblock J. Evol. Equ. 18, 411--465.(2018)

%
\end{thebibliography}
\end{document}